\documentclass{article}
\usepackage[fontsize=12pt]{scrextend}
\usepackage{babel}
\usepackage{amsmath}
\usepackage{mathtools}
\usepackage{euscript, amsmath,amssymb,amsfonts,mathrsfs,amsthm,mathtools,graphicx, tikz, xcolor,verbatim, bm, enumerate, enumitem,multicol,appendix,etoolbox}
\usepackage{wrapfig}
\usepackage[all]{xy}
\usepackage{upquote}
\usepackage{listings}
\usetikzlibrary{arrows,patterns}
\usepackage{authblk}

\usepackage[latin1]{inputenc}
\usepackage{verbatim}
\usepackage{bm}
\usepackage[justification=centering]{subcaption}

\lstdefinelanguage{Sage}[]{Python}
{morekeywords={True,False,sage,singular},
sensitive=true}
\definecolor{dblackcolor}{rgb}{0.0,0.0,0.0}
\definecolor{dbluecolor}{rgb}{.01,.02,0.7}
\definecolor{dredcolor}{rgb}{0.8,0,0}
\definecolor{dgraycolor}{rgb}{0.30, 0.3,0.30}

\usepackage[outer=1in,marginparwidth=.75in]{geometry}
\usepackage{marginnote}
\usetikzlibrary{calc}
\usetikzlibrary{positioning}
\usetikzlibrary{shapes.geometric}
\usetikzlibrary{shapes.geometric}

\usepackage{color}
\usepackage[latin1]{inputenc}

\tikzstyle{square} = [shape=regular polygon, regular polygon sides=4, minimum size=1cm, draw, inner sep=0, anchor=south, fill=gray!30]
\tikzstyle{squared} = [shape=regular polygon, regular polygon sides=4, minimum size=1cm, draw, inner sep=0, anchor=south, fill=gray!60]

\newtheorem{theorem}{Theorem}[section]
\newtheorem{definition}[theorem]{Definition}

\newtheorem{lemma}[theorem]{Lemma}
\newtheorem{coro}[theorem]{Corollary}
\newtheorem{example}[theorem]{Example}
\newtheorem{prop}[theorem]{Proposition}

\newcommand{\Z}{{\mathbb{Z}}}
\newcommand{\Q}{{\mathbb{Q}}}
\newcommand{\N}{{\mathbb{N}}}

\newcommand{\Gal}{{\mathrm{Gal}}}

\newcommand{\Nm}{{\mathrm{Nm}}}

\providecommand{\keywords}[1]
{
  \small	
  \textbf{\textit{Keywords---}} #1
}

\begin{document}

\title{An Elementary Proof of the  Minimal Euclidean Function on the Gaussian Integers}
\author{Hester Graves}
\affil{Center for Computing Sciences/IDA}
\date{\today}

\maketitle

\abstract{Every Euclidean domain $R$ has a minimal Euclidean function, $\phi_R$.  
A companion paper \cite{Graves} introduced a formula to compute $\phi_{\Z[i]}$.
It is the first formula for a minimal Euclidean function for 
the ring of integers of a non-trivial number field.
It did so by studying the geometry of the set $B_n = \left \{ \sum_{j=0}^n v_j (1+i)^j : v_j \in \{0, \pm 1, \pm i \} \right \}$ and 
then applied  Lenstra's result that $\phi_{\Z[i]}^{-1}([0,n]) = B_n$
 to provide a short proof of $\phi_{\Z[i]}$.  
Lenstra's proof requires s substantial algebra background.
This paper uses the new geometry of the sets $B_n$ to prove the formula for $\phi_{\Z[i]}$ without using Lenstra's result.
The new geometric method lets us prove Lenstra's theorem using only elementary methods. 
We then apply the new formula to answer Pierre Samuel's open question: what is the size of $\phi_{\Z[i]}^{-1}(n)$?.
Appendices provide a table of answers and the associated SAGE code. \\
\keywords{number theory, Euclidean algorithm, Euclidean function, Euclidean domain, Gaussian integers, quadratic number fields}

\section{Introduction}\label{introduction}
This paper presents the first formula that computes the minimal Euclidean function for a non-trivial number field.  
Theorem \ref{formula_statement} gives a formula for $\phi_{\Z[i]}$, the minimal Euclidean function for $\Z[i]$.
The ring $\Z[i]$, also called the Gaussian integers or the Gaussians, is the ring of integers of $\Q(i)$.  
Calculating the minimal Euclidean function for any number field's ring of integers (other than $\Z$, the ring of integers of $\Q$) 
has been an open problem since Motzkin introduced minimal Euclidean functions in 1941. 
Pierre Samuel explicitly mentioned being unable to generally enumerate the pre-images of $\phi_{\Z[i]}^{-1}$ in 1971 \cite{Samuel}.
Section~\ref{history} provides the question's history.

To the author's surprise, $\phi_{\Z[i]}$ is easy to compute, and can be done by hand for small examples.
Sections~\ref{expansions} and \ref{Main Result} study the geometry of the sets $\phi_{\Z[i]}^{-1}([0,n])$. 
Samuel calculated $|\phi_{\Z[i]}^{-1}(n)|$ for $n \in [0,8]$.
Section~\ref{Application} shows how to quickly compute $\phi_{\Z[i]}^{-1} (9)$,
and gives a closed form expression for $|\phi_{\Z[i]}^{-1}|$ for $n\geq 2$.
Appendix~\ref{Table} is a table of these values.   
The section also compares our new formula with the previous recursive methods to compute $\phi_{\Z[i]}^{-1}([0,n])$;
Appendix~\ref{Code} provides code for those older techniques.

A companion paper \cite{Graves} gives a short proof of Theorem \ref{formula_statement}, using a result of Lenstra.
Lenstra's proof requires comfort with a range of ideas in algebra.  
We use our new geometric description of the sets $B_n$ to provide a shorter, alternative proof of Lenstra's theorem.
This paper, therefore, provides a self-contained, elementary proof, at the expense of the brevity of \cite{Graves}.

The only background knowledge required is familiarity with complex conjugation and quotients in rings.
The proof focuses on the geometry of the sets $\phi_{\Z[i]}^{-1}([0,n])$, 
so readers will want to study the figures carefully, and pay particular attention to Figure \ref{Fig:triangle}.  

\subsection{History}\label{history}
Answering a question of Zariski, Motzkin showed in 1949 that every Euclidean domain $R$ has a unique minimal Euclidean function $\phi_R$.  
His paper only gave one example in a number field: he showed that $\phi_{\Z}(x)$ is the 
number of digits in the binary expansion of $|x|$, or $\lfloor \log_2(|x|) \rfloor$ \cite{Motzkin}.  
Following his lead, mathematicians searched fruitlessly for minimal Euclidean  functions for number fields' rings of integers.  
Pierre Samuel calculated $\phi_{\Z[i]}^{-1}(n)$ and
$\phi_{\Z[\sqrt{2}]}^{-1}(n)$ for $n\leq 8$
\footnote{Conscientious readers who check the original source will note that Samuel claimed that he went up to $n=9$.  
He used a slightly different definition, so that $\phi_{\Z[i]}(0) \neq \phi_{\Z[i]}(1)$.  
This footnoted sentence is his result, translated to our notation using Definition~\ref{construction}.},
  and said in his survey `About Euclidean Rings' that the sets were `very irregular (\cite{Samuel}, p. 290).'
He explicitly expressed interest in computing the sets, and included their various sizes.

In his monograph ``Lectures in Number Fields\cite{Lenstra}," Lenstra showed on page 49 that  
\begin{equation}\label{1+i expansion}
\phi_{\Z[i]}^{-1}([0,n]) = \left \{ \sum_{j=0}^n v_j (1+i)^j : v_j \in \{0, \pm 1, \pm i \} \right \}.
\end{equation}
Note that Lenstra, unlike Motzkin in his study of $\Z$, provided an algebraic description of the preimages of $\phi_{\Z[i]}$,
rather than a function.  That may seem like a distinction without a difference, but in the Gaussians, it is not 
easy to determine the least $n$ for which $a+bi$ can be written as a $(1+i)$-ary expansion of length $\leq n$.
Section \ref{expansions} expands on some of these challenges.  
Using Lenstra's result to compute $\phi_{\Z[i]}^{-1}(9)$ (where Samuel stopped his computation) would require computing 
$v (1+i)^9 + w$ for all possible $v_j \in \{ \pm 1, \pm i\}$ and $w \in \phi_{\Z[i]}^{-1}([0,8])$.
One would then remove any elements that appear in $\phi_{\Z[i]}^{-1}([0,8])$.  
An explicit formula allows us to directly compute the elements, without the repetition required by the recursive method outlined above. 
We see in Section~\ref{Application} that Theorem~\ref{pre-image_cardinality} calculates the cardinality of $\phi_{\Z[i]}^{-1}(n)$ for $n \geq 1$
 without enumerating all of the sets' elements.

In \cite{Graves}, the author explicitly computed $\phi_{\Z[i]}$, using the sequence $w_n$.  
We define $B_n = \left \{ \sum_{j=0}^n v_j (1+i)^j : v_j \in \{0, \pm 1, \pm i \} \right \}$, 
the Gaussians' $(1+i)$-ary analogue of the set of integers with binary expansions of length $\leq n$.
That paper gives a formula to find the least $n$ such that a Gaussian integer is an element of $B_n$.
It then uses Lenstra's theorem (Equation \ref{1+i expansion}) to show that $\phi_{\Z[i]}$ is given by that formula.
\begin{definition} For $k \geq 0$, $w_{2k} = 3 \cdot 2^k$ and $w_{2k +1} = 4 \cdot 2^k$.
\end{definition}
We denote $b$ divides $a$ by $a \mid b$.  When $b^ c \mid a$ but $b^{c+1} \nmid a$, we write $b^c \parallel a$.

\begin{theorem}\label{formula_statement} (Theorem 1.2 in \cite{Graves})
Suppose that $a+bi \in \Z[i] \setminus 0$, that $2^j \parallel a+bi$, and that $n$ is the least integer such that 
$\max \left ( \left | \frac{a}{2^j} \right |, \left | \frac{b}{2^j} \right | \right ) + 2 \leq w_n$.
If $\left | \frac{a}{2^j} \right | + \left | \frac{b}{2^j} \right | + 3 \leq w_{n+1} $, then 
$\phi_{Z[i]}(a+bi) = n + 2j$.  Otherwise, $\phi_{Z[i]}(a+bi) = n + 2j +1$.
\end{theorem}

The formula's proof in \cite{Graves} provided a geometric description of the sets $B_n$.
Section~\ref{expansions} defines the geometry used in \cite{Graves}, and uses it to study our sets $B_n$. 
Sections~\ref{expansions} and \ref{Main Result} then show that $\phi_{\Z[i]}^{-1}([0,n]) = B_n$ and thus
\[\phi_{\Z[i]}^{-1}([0,n]) \setminus 0 =    \displaystyle \coprod_{j=0}^{\lfloor n/2 \rfloor } ( a + bi: 2^j \parallel a + bi, 
\max(|a|, |b|) \leq w_n - 2^{j+1},  |a| + |b| \leq w_{n+1} - 3 \cdot 2^j \}, \]
thereby bypassing Lenstra's proof.  
We do this because Lenstra's proof requires an extensive knowledge of algebra, while this paper's arguments are elementary.

As a consequence of Theorem \ref{octo_union} in \cite{Graves} and Section~\ref{expansions}, we answer Samuel's question
by  characterizing the sets $\phi_{\Z[i]}^{-1}(n)$ and then providing a closed-form formula computing $|\phi_{\Z[i]}^{-1}(n)|$.

\begin{theorem}\label{pre-images} For $k \geq 1$,\\
$\begin{array}{ccc}
\phi_{\Z[i]}^{-1}(2k +1) & = 
&\displaystyle \coprod _{j=0}^{k}
\left ( 
a+bi:
\begin{array}{c}
2^j \parallel (a+bi); |a|, |b|\leq w_n - 2^{j+1}; \\
|a| + |b| \leq w_{n+1} - 3 \cdot 2^j ,\\
\text{ and either } \max(|a|, |b|) > w_{n-1} - 2^{j+1} \\
\text{ or } |a| + |b| > w_{n} - 3 \cdot 2^j 
\end{array}
\right ) \\
\text{and} && \\
\phi_{\Z[i]}^{-1}(2k)  & = 
&\begin{array}{c}
 \{\pm 2^k, \pm 2^k i \} \cup \\
 \displaystyle \coprod _{j=0}^{k-1}
\left ( 
a+bi:
\begin{array}{c}2^j \parallel (a+bi); |a|, |b|\leq w_n - 2^{j+1};\\
  |a| + |b| \leq w_{n+1} - 3 \cdot 2^j ,\\
\text{ and either } \max(|a|, |b|) > w_{n-1} - 2^{j+1} \\
\text{ or } |a| + |b| > w_{n} - 3 \cdot 2^j 
\end{array}
\right ).
\end{array}
\end{array}$
\end{theorem}
We use this description to find the following expressions.
 \begin{theorem}\label{size_of_sets}
 For $k\geq 1$, 
 \begin{align*}
 |\phi_{\Z[i]}^{-1} (2k)| &= 14 \cdot 4^k - 14 \cdot 2^k + 4\\
 \intertext{ and} 
 |\phi_{\Z[i]}^{-1}(2k +1)| &= 28 \cdot 4^k - 20 \cdot 2^k + 4.
 \end{align*}
 \end{theorem}
Appendix \ref{Table} is a table of the values of $|\phi_{\Z[i]}^{-1} (n)|$.

\section{Preliminaries}

\subsection{Motzkin's Lemma and minimal Euclidean functions}

A domain $R$ is \textbf{Euclidean} if there exists a \textbf{Euclidean function} $f$,
$f: R \setminus 0 \rightarrow \N,$ such that if $a \in R$ and $b \in R \setminus 0$, then there exist some $q,r \in R$ such that $a =qb +r$, 
where either $r=0$ or $f(r) < f(b)$.\footnote{Motzkin and Lenstra both define $f: R \setminus 0 \rightarrow W$, where $W$ is a 
well-ordered set with $\N$ as an initial segment.}

We can restate this standard definition of Euclidean functions in terms of cosets, by saying that $f:R \setminus 0 \rightarrow \N$ is a Euclidean function if, for all $b \in R \setminus 0$,
 every non-zero coset $[a] \in R/b$ has a representative $r$ (i.e., $a \equiv r \pmod {b}$) such that $f(r) < f(b)$.  This reformulation paves the way for Motzkin's Lemma.

\begin{definition}\label{construction} \textbf{Motzkin Sets} \cite{Motzkin} Given a domain $R$, define 
\begin{align*}
A_{R,0} &: = 0 \cup R^{\times} \\
A_{R,j} &: = A_{R, j-1} \cup \{ \beta :A_{R,j-1} \twoheadrightarrow R/\beta \}, \text{ and}\\ 
A_R & := \bigcup_{j=0}^{\infty} A_{R,j},
\end{align*}
where $R^{\times}$ is the multiplicative group of $R$ and $G \twoheadrightarrow R/ \beta$ if every $[a] \in R/\beta$ has a representative $r \in G$.
\end{definition}

Studying $A_{\Z}$ clarifies this cumbersome definition.  The elements $[0]$, $[1]$, and $[2]$ of $\Z / 3\Z$ 
can be represented as $[0]$, $[1]$, and $[-1]$, as 
$2 \equiv -1 \pmod{3}$. 

\begin{example}\label{example_in_Z}  When $R = \Z$, our Motzkin sets are 
\begin{align*}
A_{\Z,0} & = \{0, \pm 1\} \\
A_{\Z,1} & = \{0, \pm 1, \pm 2, \pm 3\} \\
A_{\Z,2} & = \{0, \pm 1, \pm 2, \pm 3, \pm 4, \pm 5, \pm 6, \pm 7\} \\
A_{\Z,n} & = \{0, \pm 1, \ldots , \pm (2^{n+1} -1)\} \\
A_{\Z} & = \Z.
\end{align*}
\end{example}

Motzkin' sets allow us to present his foundational lemma.

\begin{lemma}(Motzkin's Lemma \cite{Motzkin}) \label{Motzkins_Lemma}  A domain $R$ is Euclidean if and only if $R = A_R$.  Furthermore, if $R$ is Euclidean, if $F$ is the set of all Euclidean functions on $R$, and if 
\begin{align*}
\phi_R &: R \setminus 0 \rightarrow \N,\\ 
\phi_R(a) &:= j \text{ if }a \in A_{R,j} \setminus A_{R, j-1},
\end{align*}
then $\phi_R(a) = \displaystyle \min_{f\in F} f(a)$ and $\phi_R$ is itself a Euclidean function.
\end{lemma}

We call $\phi_R$ the \textbf{minimal Euclidean function} on $R$.  Example \ref{example_in_Z} shows that $\phi_{\Z} (x) = \lfloor \log_2 |x| \rfloor$ is the number of digits in the binary expansion of 
$x$, as mentioned in the introduction.
Before Motzkin's Lemma, proving a domain was Euclidean was an exercise in trial and error, as people searched for potential Euclidean functions.  
Motzkin showed that if a Euclidean function exists, then the Motzkin sets explicitly define it.  
Motzkin's Lemma tells us that $A_{R, n} = \phi_{R}^{-1} ([0,n])$.

The simplest applications of Motzkin's Lemma show that certain rings are not Euclidean.
If $R$ is a principal ideal domain with finitely many multiplicative units, 
it is easy to compute $A_{R,n}$ for small $n$.  If the sets stabilize, then $A_R \subsetneq R$ and $R$ is not a Euclidean domain.
Computing Motzkin sets quickly shows that 
while $\Q(\frac{1 + \sqrt{-19}}{2})$ is principal, it is not Euclidean.

\subsection{Motzkin Sets for the Gaussian Integers}\label{A_sets}

The elements of $\Z[i] = \{ a + bi: a, b \in \Z \}$ are called Gaussian integers because Gauss showed that 
$\Nm(a+bi) = a^2 + b^2$ is a Euclidean function for $\Z[i]$, making $\Z[i]$ a norm-Euclidean ring.
The (algebraic) norm is a multiplicative function, so $\Nm(a+bi) \Nm(c+di) = \Nm((a+bi)(c+di))$, and $\Nm(a+bi) = |\Z[i]/(a+bi)\Z[i]|$, the number of cosets of $a+bi$.  
The domain $\Z[i]$ is the ring of integers of $\Q(i)$, and its group of multiplicative units is $\Z[i]^{\times} = \{ \pm 1, \pm i \}$.  
Following Definition \ref{construction}, we present the first three Motzkin sets for $\Z[i]$.
\begin{example}\label{example_in_G}
\begin{align*}
A_{\mathbb{Z}[i], 0} &= \{0, \pm 1, \pm i \},\\
A_{\mathbb{Z}[i], 1} & =  \{0, \pm 1, \pm i , \pm 1 \pm i, \pm 2 \pm i, \pm 1 \pm 2i\},\\
A_{\mathbb{Z}[i], 2} & =  \{0, \pm 1, \pm i , \pm 1 \pm i, \pm 2 \pm i, \pm 1 \pm 2i\}  \\
&  \cup  \{  \pm 2, \pm 2i, \pm 3, \pm 3i, \pm 3 \pm i, \pm 1 \pm 3i, \pm 4 \pm i, \pm 1 \pm 4i, \pm 2 \pm 3i, \pm 3 \pm 2i\}.
\end{align*}
\end{example}
For $n \geq 1$,
\[A_{\mathbb{Z}[i],n} = A_{\mathbb{Z}[i],n-1} \cup \{a+bi \in \mathbb{Z}[i] :A_{\mathbb{Z}[i], n-1} \twoheadrightarrow \mathbb{Z}[i]/(a+bi) \},\]
so the sets $A_{\mathbb{Z}[i], n}$ are closed under multiplication by units, as $a+bi$ and its
associates $u(a+bi)$, $u \in \Z[i]^{\times}$, generate the same ideal.  This gives the sets $A_{\mathbb{Z}[i], n}$ a four-fold symmetry,
but the Gaussian integers' Motzkin sets actually have an eight-fold symmetry.

\begin{lemma}\label{cc}  The sets $A_{\mathbb{Z}[i],n}$ are closed under complex conjugation.
\end{lemma}
\begin{proof} We use induction; note that $A_{\mathbb{Z}[i],0}$ is closed under complex conjugation.  
Suppose that $A_{\mathbb{Z}[i],n}$ is closed under complex conjugation, that $a+bi \in A_{\mathbb{Z}[i], n+1}$, and that 
$[x] \in \mathbb{Z}[i] / (\overline{a+bi})$.  Then there exist some $q$ in 
$\mathbb{Z}[i]$ and some $r \in A_{\mathbb{Z}[i], n}$ such that 
$\overline{x} = q (a+bi) + r$.
Our induction hypothesis forces $\overline{r}$ to be an element of $A_{\mathbb{Z}[i], n}$,  and as 
$x = \overline{q} (\overline{a+bi} ) + \overline{r}$, $A_{\Z[i],n} \twoheadrightarrow \Z/(\overline{a+bi})\Z$ and  
$\overline{a+bi} \in A_{\mathbb{Z}[i], n+1}$.
\end{proof}

\begin{coro}\label{you_get_the_whole_set} An element $a+bi \in A_{\mathbb{Z}[i],n}$ if and only if $\{ \pm a \pm bi \}, \{ \pm b \pm ai\} \subset A_{\mathbb{Z}[i],n}$.
\end{coro}

Lemma \ref{cc} is a special case of the general result that if $K$ is a Galois number field, its Motzkin sets are closed under $\sigma$ for all $\sigma \in \Gal(K/ \Q)$.

\subsection{Representatives of Cosets of $a+bi$}\label{cosets}

Our definition of $A_{\Z[i],n}$ relies on sets that surject onto quotients $\Z[i]/(a + bi)$, so it behooves us to study how subsets of $\Z[i]$ map onto these quotients.  
First, we examine squares in the plane.  


\begin{lemma}\label{a_square}  
When $a > b \geq 0$, distinct elements in an $a \times a$ square in $\Z[i]$ are not congruent modulo $a +bi$. In other words, if $a > b \geq 0$, if $c,d \in \mathbb{Z}$,  if 
\begin{equation*}
S = \{ x+yi: c \leq x < c +a, d \leq y < d + a\},
\end{equation*}
and if $\alpha + \beta i, \gamma + \delta i$ are distinct elements of $S$, then $\alpha + \beta i \not \equiv \gamma +\delta i \pmod{a + bi}$.
\end{lemma}
\begin{proof} Suppose, leading to a contradiction, that $\alpha + \beta i \equiv \gamma +\delta i \pmod{a+bi}$.  Then there exists some $y \in \mathbb{Z}[i]$ such that 
$(\alpha - \gamma) + (\beta -\delta) i = y (a+bi)$.  
Note that 
\begin{equation*}
\Nm(y) \Nm(a+bi) = (\alpha -\gamma)^2 + (\beta -\delta)^2 \leq 2(a-1)^2 < 2(a^2 + b^2)=2 \Nm(a+bi).
\end{equation*}
As $\alpha + \beta i \neq \gamma + \delta i$, the norm of $y$ equals one, so  
$(\alpha - \gamma) + (\beta -\delta)i \in \{ \pm (a+bi), \pm (b-ai)\}$,
which cannot be, as $|\alpha -\gamma|, |\beta -\delta| \leq a-1$.
\end{proof} 

\begin{lemma} \label{two_squares} 
 If $a > b \geq 0$, if 
$S = \{ x+yi: 0 \leq x,y < a\}$, if $T=\{ x+iy: 0 \leq x <b, -b \leq y <0\}$,
and if $\alpha + \beta i, \gamma + \delta i$ are distinct elements of any translate of $S \cup T$, then 
$\alpha + \beta i \not \equiv \gamma + \delta i \pmod{a +bi}$ and $|S \cup T| = \Nm(a +bi)$.
The set $S \cup T$ contains exactly one representative of every coset of $a+bi$.
\end{lemma}

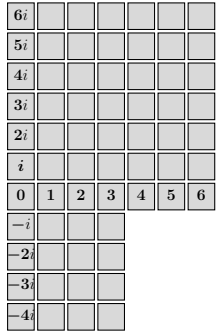
\begin{figure}[ht] 
\centering 

	\begin{tikzpicture} [scale=.5, transform shape]						
		
		\foreach \x in {0,...,6}
		\foreach \y in {0,...,6}{%
			\node[square]  at (.8*\x,.8*\y) {}; }
			
		\foreach \x in {0,...,3}
		\foreach \y in {1,...,4}{%
			\node[square]  at (.8*\x,-.8*\y) {}; }
			
		\foreach \x in {0,...,6}
			\node[circle,minimum size=1cm]  at (.8*\x,.4) {$\bm \x $}; 
			
		\foreach \y in {-4,...,-2}
			\node[circle,minimum size=1cm]  at (0,.4 + .8*\y) {$\bm \y i $}; 
			
		\node[circle,minimum size=1cm]  at (0,-.4) {$\bm -i $}; 
		\node[circle,minimum size=1cm]  at (0,1.2) {$\bm i $};	
			
		\foreach \y in {2,...,6}
			\node[circle,minimum size=1cm]  at (0,.4 + .8*\y) {$\bm \y i $};

	\end{tikzpicture}
	\caption{$S \cup T$ for $a +bi = 7 +4i$}	
	\label{Fig:S_cup_T}			
\end{figure}

\begin{proof} See Figure \ref{Fig:S_cup_T}.  Lemma \ref{a_square} shows that two distinct elements of $S$ (respectively, $T$) are not equivalent modulo $a+bi$.  It remains to show that if 
$\alpha + \beta i \in T$ and $\gamma + \delta i \in S$, then $\alpha + \beta i \not \equiv \gamma + \delta i \pmod{a+bi}$.  

Suppose, leading to a contradiction, there exists some 
$y \in \mathbb{Z}[i]$ such that $(\alpha + \beta i) - (\gamma + \delta i) = y(a+bi)$.  Then
\begin{align*}
\Nm(y)\Nm(a+bi) &= (\alpha - \gamma)^2 + (\beta - \delta)^2 \\
& \leq (a-1)^2 + (a+b-1)^2\\
& < 4 (a^2 + b^2) = 4 \Nm(a+bi),
\end{align*}
 so $1 \leq \Nm(y) < 4$. This means that $\Nm(y) = 1$ or $2$, as there are no Gaussian integers with norm $3$.  The Gaussian integers 
 with norm $1$ or $2$ are $\{ \pm 1, \pm i, \pm 1 \pm i \}$ 
and thus the set $C$ of potential values of $y(a+bi)$, where the real part of $y(a+bi)$ is $\geq 0$, is 
\begin{equation*}
 \{  a+bi,  b-ai,  a-b + (a+b)i,  a + b + (b-a)i \}.
\end{equation*}
If $x \in C$, if $\alpha + \beta i \in S$, and if $\gamma + \delta i \in T$, then neither $x + \alpha + \beta i$ nor 
$x + \gamma + \delta i$ is in $S \cup T$ (see Figure \ref{Fig:triangle}), so no two distinct elements of $S \cup T$ are congruent modulo $a +bi$.  
As $S$ and $T$ are disjoint, as $|S| = a^2$, and as $|T| = b^2$,
 the size of their union is $|S \cup T |= a^2 + b^2 = \Nm(a +bi)$.  
 We conclude that any translate of $S \cup T$ contains precisely one representative for each coset of $a +bi$.
\end{proof}

\begin{coro}\label{down_to_one_square} If $M \subset \Z[i]$, if $M$ is closed under multiplication by units, and if 
$S \subset U = \displaystyle \bigcup_{q \in \Z[i]} ( M + q(a +bi))$, then $M \twoheadrightarrow \Z[i]/(a+bi)$.
\end{coro}

\begin{proof}  If $M$ is closed under multiplication by units and $S \subset U$, then $T \subset -iS \subset -i U \subset U$, and $S \cup T \subset U$.  

Given $[x] \in \Z[i]/(a +bi)$,  there exists an $r \in (S \cup T)$ such that 
$[x] = [r]$ by Lemma \ref{two_squares}.  Our hypothesis says there exist an $m \in M$ and $q \in \Z[i]$ such that $r = m + q(a +bi)$.  We conclude that 
$[m] = [x]$ and thus  $M \twoheadrightarrow \Z[i]/(a +bi)$.
\end{proof}

So far, we have looked at squares 
to analyze collections of representatives of cosets of $a +bi$.  
We now turn to triangles.

\begin{definition} \label{basic_triangle}  If $a+bi \in \Z[i] \setminus 0$, let
\begin{equation*}
\mathscr{S}_{a+bi} := \{ x+yi: 0 \leq x,y, x +y  < \max (|a|, |b| )\}.
\end{equation*}
\end{definition}

\begin{lemma}\label{triangle} Suppose that  $a > b \geq 0$, that $(1 +i)  \nmid a +bi$, and that $M \subset \Z[i]$ is closed under multiplication by units.  
If 
$\mathscr{S}_{a+bi} \subset U = \displaystyle \bigcup_{q \in \Z[i]} (M + q(a+bi))$, then $M \twoheadrightarrow \mathbb{Z}[i]/(a+bi)$.

\end{lemma}

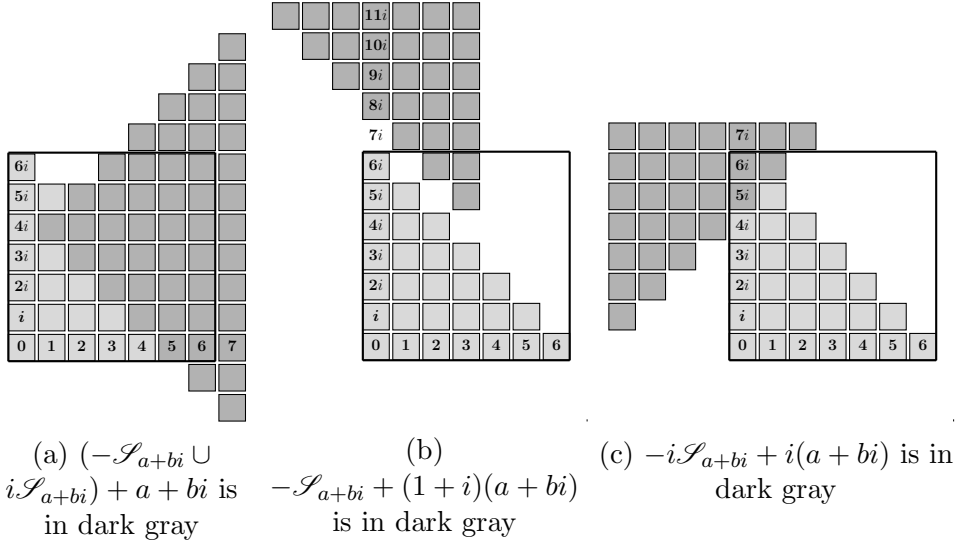
\begin{figure}[ht]\centering

\subcaptionbox{
 $(- \mathscr{S}_{a+bi} \cup i \mathscr{S}_{a+bi} ) + a +bi$ is in dark gray}{
	\begin{tikzpicture} [scale=.5, transform shape]						
		
		\foreach \y in {0,...,6}
		\node[square]  at (0,.8*\y) {};

		\foreach \y in {0,...,3}
		\node[square]  at (.8,.8*\y) {};
		\foreach \y in {5,...,5}
		\node[square]  at (.8,.8*\y) {};
		\foreach \y in {4,...,4}
		\node[squared]  at (.8,.8*\y) {};
		
		\foreach \y in {0,...,2}
		\node[square]  at (1.6,.8*\y) {};
		\foreach \y in {3,...,5}
		\node[squared]  at (1.6,.8*\y) {};
		
		\foreach \y in {0,...,1}
		\node[square]  at (2.4,.8*\y) {};
		\foreach \y in {2,...,6}
		\node[squared]  at (2.4,.8*\y) {};
		
		\node[square]  at (3.2,0) {};
		\foreach \y in {1,...,7}
		\node[squared]  at (3.2,.8*\y) {};
		
		\foreach \y in {0,...,8}
		\node[squared]  at (4,.8*\y) {};
		
		\foreach \y in {-1,...,9}
		\node[squared]  at (4.8,.8*\y) {};
		
		\foreach \y in {-2,...,10}
		\node[squared]  at (5.6,.8*\y) {};

		\foreach \x in {0,...,7}
			\node[circle,minimum size=1cm]  at (.8*\x,.4) {$\bm \x $};

		\node[circle,minimum size=1cm]  at (0,1.2) {$\bm i $};	
			
		\foreach \y in {2,...,6}
			\node[circle,minimum size=1cm]  at (0,.4 + .8*\y) {$\bm \y i $}; 		
			
		\draw[thick] (-.35,0)--(5.15,0);
		\draw[thick] (-.35,5.55)--(5.15,5.55);
		\draw[thick] (-.35,0)--(-.35,5.55);	
		\draw[thick] (5.15,0)--(5.15,5.55);	
		

	\end{tikzpicture}}		
	\subcaptionbox{
 $- \mathscr{S}_{a+bi}  + (1 +i)(a +bi)$ is in dark gray}{
	\begin{tikzpicture} [scale=.5, transform shape]						
		
		\foreach \y in {11,...,11}
		\node[squared]  at (-2.4,.8*\y) {};
		
		\foreach \y in {10,...,11}
		\node[squared]  at (-1.6,.8*\y) {}; 
		
		\foreach \y in {9,...,11}
		\node[squared]  at (-.8,.8*\y) {}; 
		
		\foreach \y in {0,...,6}
		\node[square]  at (0,.8*\y) {}; 
		
		\foreach \y in {8,...,11}
		\node[squared]  at (0,.8*\y) {};

		\foreach \y in {0,...,5}
		\node[square]  at (.8,.8*\y) {};
		\foreach \y in {7,...,11}
		\node[squared]  at (.8,.8*\y) {};
		
		\foreach \y in {0,...,4}
		\node[square]  at (1.6,.8*\y) {};
		\foreach \y in {6,...,11}
		\node[squared]  at (1.6,.8*\y) {};
		
		\foreach \y in {0,...,3}
		\node[square]  at (2.4,.8*\y)  {};
		\foreach \y in {5,...,11}
		\node[squared]  at (2.4,.8*\y)  {};
		
		\foreach \y in {0,...,2}
		\node[square]  at (3.2,.8*\y)  {};
		
		\foreach \y in {0,...,1}
		\node[square]  at (4,.8*\y)  {};
		
		\node[square]  at (4.8,0)  {};

		\foreach \x in {0,...,6}
			\node[circle,minimum size=1cm]  at (.8*\x,.4) {$\bm \x $};

		\node[circle,minimum size=1cm]  at (0,1.2) {$\bm i $};	
			
		\foreach \y in {2,...,11}
			\node[circle,minimum size=1cm]  at (0,.4 + .8*\y) {$\bm \y i $}; 		
			
		\draw[thick] (-.35,0)--(5.15,0);
		\draw[thick] (-.35,5.55)--(5.15,5.55);
		\draw[thick] (-.35,0)--(-.35,5.55);	
		\draw[thick] (5.15,0)--(5.15,5.55);	
		
		\draw[thick] (5.6, -1.6) --(5.6, -1.6);			
		
	\end{tikzpicture}}		
	\subcaptionbox{
 $-i \mathscr{S}_{a+bi}  + i(a +bi)$ is in dark gray}{
	\begin{tikzpicture} [scale=.5, transform shape]	
	
		\foreach \y in {1,...,7}
		\node[squared]  at (-3.2,.8*\y) {};
		
		\foreach \y in {2,...,7}
		\node[squared]  at (-2.4,.8*\y) {};
		
		\foreach \y in {3,...,7}
		\node[squared]  at (-1.6,.8*\y) {};
		
		\foreach \y in {4,...,7}
		\node[squared]  at (-.8,.8*\y) {};

		\foreach \y in {0,...,4}
		\node[square]  at (0,.8*\y) {}; 
		\foreach \y in {5,...,7}
		\node[squared]  at (0,.8*\y) {}; 		
		
		\foreach \y in {0,...,5}
		\node[square]  at (.8,.8*\y) {};
		\foreach \y in {6,...,7}
		\node[squared]  at (.8,.8*\y) {};
		
		\foreach \y in {0,...,4}
		\node[square]  at (1.6,.8*\y) {};
		\foreach \y in {7,...,7}
		\node[squared]  at (1.6,.8*\y) {};
		
		\foreach \y in {0,...,3}
		\node[square]  at (2.4,.8*\y) {};
		
		\foreach \y in {0,...,2}
		\node[square]  at (3.2,.8*\y) {};
		
		\foreach \y in {0,...,1}
		\node[square]  at (4,.8*\y) {};
		
		\foreach \y in {0,...,0}
		\node[square]  at (4.8,.8*\y) {};

		\foreach \x in {0,...,6}
			\node[circle,minimum size=1cm]  at (.8*\x,.4) {$\bm \x $};

		\node[circle,minimum size=1cm]  at (0,1.2) {$\bm i $};	
			
		\foreach \y in {2,...,7}
			\node[circle,minimum size=1cm]  at (0,.4 + .8*\y) {$\bm \y i $}; 		
		
		\draw[thick] (-.35,0)--(5.15,0);
		\draw[thick] (-.35,5.55)--(5.15,5.55);
		\draw[thick] (-.35,0)--(-.35,5.55);	
		\draw[thick] (5.15,0)--(5.15,5.55);	
		
		\draw[thick] (5.6, -1.6) --(5.6, -1.6);		
		
	\end{tikzpicture}}
\caption{When $a +bi = 7 +4i$\\$\mathscr{S}_{a+bi}$ is in light gray in all three figures}
\label{Fig:triangle}				
\end{figure}

\begin{proof}  
We will show that if $\mathscr{S}_{a+bi} \subset U$, then $S = \{ x +yi: 0 \leq x, y <a \}$ is also contained in $U$.
  Observe that if $u \in \{ \pm 1, \pm i\}$, if $q \in \Z[i]$,
and if $\mathscr{S}_{a+bi} \subset U$, then $u (\mathscr{S}_{a+bi} + q(a+bi)) \subset U$.  
Figure \ref{Fig:triangle}, with its outlined $S$, may help the reader visualize the following arguments.

Computation shows that
\begin{equation}\label{long}
 ((- \mathscr{S}_{a+bi} \cup i \mathscr{S}_{a+bi}) + a + bi) \supset \{x + yi: 0 < x \leq a, -x + b < y < x + b \}.
 \end{equation}
The set $\mathscr{S}_{a+bi}$ can be written as $\{x + yi: 0 \leq x <a, 0 \leq y<a-x\}$.  As $a >b$, $-x + b < a-x$ for all $x$ and thus 
equation \ref{long} implies that 
\begin{align}\label{triangle_subsets}
\nonumber
U &\supset \mathscr{S}_{a+bi} \cup ((- \mathscr{S}_{a+bi} \cup i \mathscr{S}_{a+bi}) + a + bi) \\
&\supset \{ x + yi: 0 \leq x < a, 0 \leq y < \max (a -x, x + b )\}.
\end{align}  
Because $x + b -1 \geq a-1$ when $x \geq a-b$, 
$\{x + yi: a-b \leq x < a, 0 \leq y < a \} \subset U$
(in Figure \ref{Fig:triangle}, this is $[3,6] \times [0, 6i] \subset U$).  Our proof that $S \subset U$ then reduces to demonstrating that
\[\{x + yi: 0 \leq x < a-b, \max (a-x, x+b ) \leq y < a \} \subset U.\]  

Mark that \[-\mathscr{S}_{a+bi} + (1+i)(a+bi) \supset \{x+yi: 0 \leq  x \leq a-b, a - x < y \leq a+b\},\] 
so $U$ contains 
$\{x + yi: 0 \leq x < a-b, 0 \leq y < a, y \neq a-x\}$.  
When $x > \frac{a-b}{2}$, $a-x < x+b $, so 
$U$ contains $\{x +yi: \frac{a-b}{2} < x < a-b, y = a-x\}$ by equation \ref{triangle_subsets}.
We have now reduced the problem to showing that 
\begin{equation} \label{diagonal_subset}
\left \{x+yi: 0 \leq x < \frac{a-b}{2}, y = a-x  \right \} \subset U;
\end{equation}
the condition is $x < \frac{a-b}{2}$ as $1+i \nmid a+bi$, which is equivalent to $a-b$ being odd.  
The variable $x$ represents an integer, so if $x \leq \frac{a-b}{2}$, then $x < \frac{a-b}{2}$.

To finish, note that 
\[-i\mathscr{S}_{a+bi} + i(a+bi) \supseteq \{x +yi: 0 \leq x < a-b, b + x < y \leq a\}.\]  
When $0 \leq x < \frac{a-b}{2}$, $a - x > b+x$,
so $-i\mathscr{S}_{a+bi} + i(a+bi)$ ( and thus the union $U$) contains $\{x+yi: 0 \leq x <\frac{a-b}{2}, y = a-x\}$.

  We have now shown that equation \ref{diagonal_subset} does hold, so
 $U$ contains all of $S$, and therefore $M \twoheadrightarrow \Z[i]/(a + bi)$
by Corollary \ref{down_to_one_square}. 

\end{proof}

\subsection{$(1 + i)$-ary expansions in $\mathbb{Z}[i]$}\label{expansions}

\begin{definition}\label{sets B_n} The sets $B_n$ are the Gaussian integers that can be written with $n+1$ `digits,' i.e. 
$$B_n = \left \{ \sum_{j=0}^n v_j (1+i)^n, v_j \in \{0, \pm 1, \pm i\} \right \}.$$
\end{definition}

This new notation allows us to restate Lenstra's result, Equation \ref{1+i expansion}, as 
$\phi_{\Z[i]}^{-1} ([0,n]) = A_{\Z[i],n} = B_n$.  
  Unfortunately for us,
it is not obvious which sets $B_n$ a given element $a+bi$ belongs to.  For example, 
as $4=-(1+i)^4$, it is clear that $4+i = -(1+i)^4 +i$, and thus $4+i \in B_4$.  It is not so obvious that 
$4+i = i(1+i)^2 +(1+i) +1,$
revealing that $4+i$ is also in $B_2$ (and thus also $B_3$).


In \cite{Graves}, the author introduced the following geometric sets and theorem, giving a fast way to compute $\phi_{\Z[i]}(a+bi)$. 
The sets are all octagonal when plotted in $\Z \times  \Z i$, as shown in Figure \ref{fig:oct_examples}.  

\begin{definition}\label{octogons}  We define 
\begin{align*}
Oct_n &: = \{ x+yi \in \Z[i]: |x|,|y| \leq w_n -2 ,|x| + |y| \leq w_{n+1} - 3 \},\\
S_n &: = \{ x+yi \in \Z[i] \setminus 0:  |x|,|y| \leq w_n -2, |x| + |y| \leq w_{n+1} - 3 ,2 \nmid \gcd (x,y)\},\\
\intertext{and}
D_n &: = \{ x+yi \in \Z[i] \setminus 0:  |x|,|y| \leq w_n -2, |x| + |y| \leq w_{n+1} - 3 ,2 \nmid (x+y)\}.
\end{align*}
\end{definition}

It follows that 
$S_n = \{x +yi \in Oct_n: (1 +i)^2 \nmid (x +yi)\}$ and 
$D_n = \{x +yi \in Oct_n: (1+i) \nmid (x+yi) \}$, so $D_n \subset S_n \subset Oct_n$, as shown in Figure \ref{fig:oct_examples}.    
Lemma 2.6 from \cite{Graves} shows that for $n \geq 1$, $S_n = D_n \cup (1+i) D_{n-1}$.  

\begin{figure}[ht]\centering
\subcaptionbox{$D_2$}{
	\begin{tikzpicture} [scale=.4, transform shape]

		\foreach \y in {-3,-1, 1,3}
		\node[square]  at (0,.8*\y) {};
		
		\foreach \y in {-2,...,2}
		\node[square]  at (.8,1.6*\y) {};
		
		\foreach \y in {-2,...,2}
		\node[square]  at (-.8,1.6*\y) {};
		
		\foreach \y in {-3,-1, 1,3}
		\node[square]  at (1.6,.8*\y) {};
		
		\foreach \y in {-3,-1, 1,3}
		\node[square]  at (-1.6,.8*\y) {};
		
		\foreach \y in {-1,...,1}
		\node[square]  at (2.4,1.6*\y) {};
		
		\foreach \y in {-1,...,1}
		\node[square]  at (-2.4,1.6*\y) {};
		
		\node[square]  at (3.2,.8) {};
		
		\node[square]  at (-3.2,.8) {};
		
		\node[square]  at (3.2,-.8) {};
		
		\node[square]  at (-3.2,-.8) {};
		
		\node [circle,minimum size=1cm] at (0,.4) {$\bm 0 $};
		\node [circle,minimum size=1cm] at (.8,.4) {$\bm 1 $};
		\node [circle,minimum size=1cm] at (-.8,.4) {$\bm -1 $};
		\node [circle,minimum size=1cm] at (0,1.2) {$\bm i $};
		\node [circle,minimum size=1cm] at (0,-.4) {$\bm -i $}; 
		 
	\end{tikzpicture}}	
\subcaptionbox{$S_2$}{
	\begin{tikzpicture} [scale=.4, transform shape]

		\node[square]  at (.8,0) {};  
		\node[square]  at (-.8,0) {}; 
		\node[square]  at (0,.8) {}; 
		\node[square]  at (0,-.8) {};

		\node[square]  at (.8, .8) {}; 
		\node[square]  at (-.8, .8) {};
		\node[square]  at (-.8, -.8) {};
		\node[square]  at (.8, -.8) {};
		
		\node[square]  at (0, 2.4) {}; 
		
		\node[square]  at (.8, 1.6) {};
		\node[square]  at (.8, 2.4) {};
		\node[square]  at (.8, 3.2) {};
		
		\node[square]  at (1.6, .8) {}; 
		\node[square]  at (1.6, 2.4) {};
		
		\node[square]  at (2.4, .8) {};
		\node[square]  at (2.4, 1.6) {};
		
		\node[square]  at (3.2, .8) {};
		
		\node[square]  at (2.4, 0) {}; 
		
		\node[square]  at (0, -2.4) {}; 
		
		\node[square]  at (.8, -1.6) {};
		\node[square]  at (.8, -2.4) {};
		\node[square]  at (.8, -3.2) {};
		
		\node[square]  at (1.6, -.8) {}; 
		\node[square]  at (1.6, -2.4) {};
		
		\node[square]  at (2.4, -.8) {};
		\node[square]  at (2.4, -1.6) {};
		
		\node[square]  at (3.2, -.8) {};
		
		\node[square]  at (0, 2.4) {}; 
		
		\node[square]  at (-.8, 1.6) {};
		\node[square]  at (-.8, 2.4) {};
		\node[square]  at (-.8, 3.2) {};
		
		\node[square]  at (-1.6, .8) {}; 
		\node[square]  at (-1.6, 2.4) {};
		
		\node[square]  at (-2.4, .8) {};
		\node[square]  at (-2.4, 1.6) {};
		
		\node[square]  at (-3.2, .8) {};
		
		\node[square]  at (-2.4, 0) {}; 
		
		\node[square]  at (-.8, -1.6) {};
		\node[square]  at (-.8, -2.4) {};
		\node[square]  at (-.8, -3.2) {};
		
		\node[square]  at (-1.6, -.8) {}; 
		\node[square]  at (-1.6, -2.4) {};
		
		\node[square]  at (-2.4, -.8) {};
		\node[square]  at (-2.4, -1.6) {};
		
		\node[square]  at (-3.2, -.8) {};
		
		\node[square]  at (0, -.8) {};
		
		\node [circle,minimum size=1cm] at (0,.4) {$\bm 0 $};
		\node [circle,minimum size=1cm] at (.8,.4) {$\bm 1 $};
		\node [circle,minimum size=1cm] at (-.8,.4) {$\bm -1 $};
		\node [circle,minimum size=1cm] at (0,1.2) {$\bm i $};
		\node [circle,minimum size=1cm] at (0,-.4) {$\bm -i $};

	\end{tikzpicture}}
	\subcaptionbox{$B_2 $}{
	\begin{tikzpicture} [scale=.4, transform shape]
		
		\node[square]  at (0,0) {};
		
		\foreach \y in {-3,-1, 1,3}
		\node[square]  at (0,.8*\y) {};
		
		\foreach \y in {-2,2}
		\node[square]  at (0,.8*\y) {};
		
		\foreach \y in {-4,...,4}
		\node[square]  at (.8,.8*\y) {};
		
		\foreach \y in {-4,...,4}
		\node[square]  at (-.8,.8*\y) {};
		
		\foreach \y in {-3,-1,1,3}
		\node[square]  at (1.6,.8*\y) {};
		
		\foreach \y in {-3,-1,1,3}
		\node[square]  at (-1.6,.8*\y) {};
		
		\node[square]  at (-1.6,0) {};
		\node[square]  at (1.6,0) {};
		
		\foreach \y in {-2,...,2}
		\node[square]  at (2.4,.8*\y) {};
		
		\foreach \y in {-2,...,2}
		\node[square]  at (-2.4,.8*\y) {};
		
		\foreach \y in {-1,1}
		\node[square]  at (3.2,.8*\y) {};
		
		\foreach \y in {-1,1}
		\node[square]  at (-3.2,.8*\y) {};
		
		\node [circle,minimum size=1cm] at (0,.4) {$\bm 0 $};
		\node [circle,minimum size=1cm] at (.8,.4) {$\bm 1 $};
		\node [circle,minimum size=1cm] at (-.8,.4) {$\bm -1 $};
		\node [circle,minimum size=1cm] at (0,1.2) {$\bm i $};
		\node [circle,minimum size=1cm] at (0,-.4) {$\bm -i $};

\end{tikzpicture}}
\subcaptionbox{$Oct_2$}{
	\begin{tikzpicture} [scale=.4, transform shape]

		\foreach \y in {-4,...,4}
		\node[square]  at (0,.8*\y) {};
		
		\foreach \y in {-4,...,4}
		\node[square]  at (.8,.8*\y) {};
		
		\foreach \y in {-4,...,4}
		\node[square]  at (-.8,.8*\y) {};
		
		\foreach \y in {-3,...,3}
		\node[square]  at (1.6,.8*\y) {};
		
		\foreach \y in {-3,...,3}
		\node[square]  at (-1.6,.8*\y) {};
		
		\foreach \y in {-2,...,2}
		\node[square]  at (2.4,.8*\y) {};
		
		\foreach \y in {-2,...,2}
		\node[square]  at (-2.4,.8*\y) {};
		
		\foreach \y in {-1,...,1}
		\node[square]  at (3.2,.8*\y) {};
		
		\foreach \y in {-1,...,1}
		\node[square]  at (-3.2,.8*\y) {};

		\node [circle,minimum size=1cm] at (0,.4) {$\bm 0 $};
		\node [circle,minimum size=1cm] at (.8,.4) {$\bm 1 $};
		\node [circle,minimum size=1cm] at (-.8,.4) {$\bm -1 $};
		\node [circle,minimum size=1cm] at (0,1.2) {$\bm i $};
		\node [circle,minimum size=1cm] at (0,-.4) {$\bm -i $}; 
		 
	\end{tikzpicture}}			
\caption{Examples of $D_n$, $S_n$, $B_n $, and $Oct_n$ when $n =2$}
\label{fig:oct_examples}			
\end{figure}
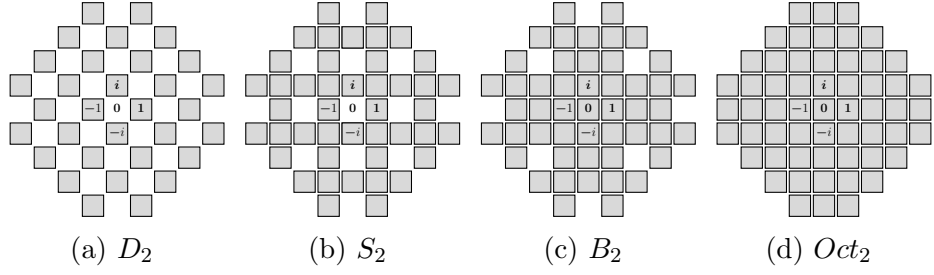
Our definitions let us describe the shape of $B_n$.
\begin{theorem}\label{octo_union} (\cite{Graves}, Theorems 2.4 and 2.7)
The set $B_n \setminus 0$ equals the disjoint union
\[  \displaystyle \coprod_{j=0}^{\lfloor n/2 \rfloor } 2^j S_{n- 2j} = \coprod_{j=0}^n (1+i)^j D_{n-j}.\]
\end{theorem}

\begin{coro}\label{one_up} Suppose that $x +yi \in Oct_n$, and that $2^l \parallel (x,y)$.  If $l \neq \lfloor \frac{n}{2} \rfloor + 1$, then $l \leq \lfloor \frac{n}{2} \rfloor$ 
and $x +yi \in B_{n+1}$.  
\end{coro}
The prove Corollary \ref{one_up}, we need the following two lemmas, which are simple to verify.

\begin{lemma}\label{max_power} If $x +yi \in Oct_n$ and $2^l \mid \gcd(x,y)$, then $l \leq \lfloor \frac{n}{2} \rfloor + 1$.  
 If $l = \lfloor \frac{n}{2} \rfloor + 1$ and $n = 2k$, 
then $x +yi \in 2^{k+1}\{ \pm 1, \pm i \}$.  
If $n = 2k +1$, then $x + yi \in 2^{k+1}\{ \pm 1, \pm i, \pm 1 \pm i \}$.
\end{lemma}
\begin{lemma}\label{identities} 
The following identities hold:
\begin{multicols}{2}
\begin{itemize}
 \item $w_{n+2} = 2 w_n$
 \item $w_{n-2} \leq w_{n+1} - w_n \leq w_{n-2}$
 \item $2(w_{n+1} - w_n) \leq w_n$
 \item $3(w_{n+1} - w_n) \leq w_{n+1}$
 \item If $2^{l+1} < w_n$, then $l\leq \lfloor \frac{n}{2} \rfloor$.
 \item If $2^{l+1} \leq w_n$, then $2^l \leq w_{n+1} - w_n$ .
 \item If $w_{n+1} - w_n \leq 2^l$, then $\lfloor \frac{n+1}{2} \rfloor \leq l$.
 \item If  $l \leq \lfloor \frac{n}{2} \rfloor$, then $2^l | (w_n - 2^l)$.
 \item If $l \leq \lfloor \frac{n}{2} \rfloor$, then $2^l \leq w_{n+1} - w_n$ .
 \item If $l \leq \lfloor \frac{n}{2} \rfloor$, then $w_{n+1} - w_n + 2^l \leq w_n$.

 \end{itemize}
 \end{multicols}
 \end{lemma}

\begin{proof} (of Corollary \ref{one_up}) If $l =0$, then $x + yi \in S_n \subset B_n \subset B_{n+1}$.  Lemma \ref{max_power} implies $l \leq \lfloor \frac{n}{2} \rfloor$,
so if $l \geq 1$, then Lemma \ref{identities} shows 
\begin{align*}
|x|, |y| & \leq w_n - 2^l = 2^l (w_{n-2l} - 1) \leq 2^l (w_{n -2l +1} -2)\\
\intertext{and}
|x| + |y| & \leq w_{n+1} - 2^l = 2^l (w_{n - 2l +1} -1) \leq 2^l( w_{n - 2l +2} -3).
\end{align*}
These equations show that $x +yi \in 2^l S_{n  - 2l +1}$ and thus, by Theorem \ref{octo_union}, also in $B_n$.
\end{proof}

The sets $B_n$ look like lacy, octagonal snowflakes, and they have several nice properties.  
Like the other sequences ($D_n$, $S_n$, and $Oct_n$), they are nested, as well as closed
under both complex conjugation and  multiplication by units.  Theorem \ref{octo_union} tells us that 
$D_n \subset S_n \subset B_n \subset Oct_n$;
Definition \ref{sets B_n} implies that if $a+bi \in B_n$, then $(1+i)^j (a+bi) \in B_{n+j}$.  Similarly, if $2^j | \gcd(a,b)$ for some $a+bi \in B_n$, then 
$\frac{a}{2^j} + \frac{b}{2^j} i \in B_{n-2j}$.
Definition \ref{sets B_n} also tells us that if $(1 +i)^{n+1} | x$ and $x \in B_{n}$, then $x =0$.  
These properties lead to the following useful result on the sets $B_n$.

\begin{lemma} \label{divides_xy} If $xy \in B_n \setminus 0$, then $x \in B_n \setminus 0$.  
\end{lemma}
\begin{proof}
Proof by induction.  
The hypothesis clearly holds for $xy \in B_0\setminus 0$, as $xy$, $x$, and $y$ are all multiplicative units, 
and $B_0 \setminus 0$ is the set of all the multiplicative units $\Z[i]^{\times}$.

Now suppose that our claim holds for all $j$, $ 0 \leq j \leq n-1$.  
Suppose that $x = a +bi$ and $y = c + di$, so $xy = (ac -bd) + (ad+bc) i \in B_n \setminus 0$. 
We will show that $x=a + bi \in B_n$. 

We may suppose that $(1+i)$ divides neither $x$ nor $y$, because then $\frac{xy}{1+i}$ would be an element of $B_{n-1}$, 
allowing us to apply our induction hypothesis.  
Corollary \ref{you_get_the_whole_set} lets us assume without loss of generality that
$a > b \geq 0$, that $c > d$, and that $a,c >0$.  There are three cases to consider.  

If $d=0$, then $0 \leq b < a \leq ac = \max (ac-bd, ad + bc) \leq w_n -2$ and 
\begin{align*}
0 &<a +b \leq ac +bc = (ac - bd) + (ad +bc) \leq w_{n+1} - 3.\\
\intertext{If $d < 0$, then }
0  &\leq b < a  \leq a +b \leq ac -bd \leq w_n - 2 \leq w_{n+1} -3.\\
\intertext{If $d >0$, then }
 0 &\leq b <a \leq a+b \leq ad+bc \leq w_n -2 \leq w_{n+1} -3.
 \end{align*}

 As $2 \nmid \gcd(a,b)$, $x = a +bi \in S_n$. 
 Theorem \ref{octo_union} tells us that $S_n \subset B_n$, so $x \in B_n$ in all three scenarios.
\end{proof}

\subsection{Motzkin sets and $(1+i)$-ary expansions}

Our proof that $A_{\mathbb{Z}[i], n} = B_n$ uses induction to show containment in both directions.
We start with three lemmas that show containment between our sets under special circumstances.  

\begin{lemma} \label{containment}If $A_{\mathbb{Z}[i], n }= B_n$, then $A_{\mathbb{Z}[i], n+1} \subset B_{n+1}$.
\end{lemma}
\begin{proof}  Given $a+bi \in A_{\mathbb{Z}[i], n+1}$, there exists some $q \in \mathbb{Z}[i]$ and $r \in A_{\mathbb{Z}[i], n}=B_n$ such that 
$(1+i)^{n+1} = q(a+bi) +r$.  Rearranging terms reveals that 
\begin{equation*}
q(a+bi) = (1+i)^{n+1} - r \in B_{n+1} \setminus 0,
\end{equation*}
so  $a+bi \in B_{n+1}$ by Lemma \ref{divides_xy}.
\end{proof}

\begin{lemma}\label{multiply_by_1+i}  If $A_{\mathbb{Z}[i], j} = B_j$ for $j \in \{n, n-1\}$, then $(1+i)B_n \subset A_{\mathbb{Z}[i], n+1}$.
\end{lemma}
\begin{proof}  Given $x \in \mathbb{Z}[i]$, we can write $x = q(1+i) +r$ for some $q \in \Z[i]$ and $r \in A_{\Z[i],0}$.  
Suppose that  $b \in B_n = A_{\mathbb{Z}[i], n}$, so we can expand $q$ as $q = q' b + r'$, where 
$r' \in A_{\Z[i], n-1}$.
Then 
\begin{align*}
(1+i)q + r &= (1+i)q' b + (1+i)r' +r\\
\intertext{and thus}
x &= q' (1+i)b + ((1+i)r'+r).
\end{align*}
The element $(1+i)r' + r \in B_n = A_{\mathbb{Z}[i], n}$, so $A_{\mathbb{Z}[i], n} \twoheadrightarrow \Z[i]/b(1+i)$ and  $b(1+i) \in A_{\mathbb{Z}[i], n+1}$.
\end{proof}

\begin{lemma} \label{subset_containment} If $A_{\Z[i], j} = B_j$ for $j \in \{n, n-1\}$, and if 
$\left ( B_{n+1} \setminus (1+i)\Z[i] \right )\subset A_{\Z[i], n+1} $, then 
$A_{\Z[i], n+1} = B_{n+1}$.
\end{lemma}
\begin{proof}  The set $B_{n+1}$ is the union of its elements that are divisible by $(1 +i)$, and the elements that are not.  The set of elements of $B_{n+1}$ that 
are divisible by $(1+i)$ is the set $(1 +i) B_n$, i.e., 
\[\{x + yi \in B_{n+1}: (1 +i) | (x +iy)\} = (1 +i) B_n.\]  Lemma \ref{multiply_by_1+i} shows that, under our assumptions,
$(1 +i)B_n \subset A_{\Z[i], n+1}$, so if $\{x + yi \in B_{n+1} : (1 +i) \nmid (x +iy)\} \subset A_{\Z[i], n+1}$, then all of $B_{n+1} \subset A_{\Z[i], n+1}$. 
Then, under our assumptions, $A_{\Z[i], n+1} \subset B_{n+1}$ by Lemma \ref{containment}, so $B_{n+1} = A_{\Z[i], n+1}$.
\end{proof}

 \section{Main Result}\label{Main Result}
We reduce proving $A_{\Z[i], n} = B_n$  to showing that $\mathscr{S}_{a+bi} \subset U = \bigcup _{q \in \Z[i]} (B_n + q(a+bi))$ for all 
$a +bi \in B_{n+1}\setminus (1+i)\Z[i]$.  
We use the geometry of our sets $D_n$, $S_n$, $B_n$, and $Oct_n$ to prove containment. 

Section \ref{iden} introduces some necessary lemmas, and Section \ref{meat} uses them to prove two technical propositions that allow us to apply 
Proposition \ref{subset_containment}.
Each of the two propositions has a long proof, broken up into cases.  
Having done all the heavy lifting,
we conclude with a short argument in subsection \ref{finally} that $A_{\Z[i], n} = B_n$.  

\subsection{Necessary Lemmas}\label{iden}

\begin{lemma}\label{oct_translate} Suppose that $a + bi \in \Z[i]\setminus (1+i)\Z[i]$ and that $u \in \mathbb{Z}[i]^{\times} = \{\pm 1, \pm i\}$.
If  $x+yi \in (Oct_n + u(a+bi))$ and $2|(x +y)$, then $x+yi \in (B_n + u(a+bi))$.
\end{lemma}
\begin{proof}  If $x+yi \in (Oct_n + u(a+bi))$, then $c +di = (x +yi) - u(a+bi)$ is an element of $Oct_n$.  
Because $(1+i) | (x +yi)$ and $(1 +i) \nmid (a+bi)$, we see that $(1+i) \nmid (c+di)$ and thus $c+di \in D_n \subset B_n$ by Theorem \ref{octo_union}.
\end{proof}

\begin{lemma}\label{broom}  Suppose that $(1+i) \nmid (a+bi)$ and that $2^k \parallel \gcd(x,y), k \geq 1$.  
If any one of $(a-x) + (b-y)i$, $(a-y) + (b+x)i$, or $-(b+x) + (a-y)i \in Oct_n$, then 
$x+yi \in U = \bigcup_{q\in \Z[i]} (B_n + q(a+bi))$.
\end{lemma}
\begin{proof}  
As $(1+i) | (x+yi)$ and $(1+i) \nmid (a+bi)$, $(1+i)$ divides neither 
$(a-x) + (b-y)i = (a+bi) - (x+yi)$ nor $-(b+x) + (a-y)i = i(a+bi) - (x+yi)$.  
It certainly does not divide
$(a-y) + (b+x)i = (a + bi) +i(x +yi).$  These three elements are all in $Oct_n \setminus (1+i) \Z[i] $, i.e., they are in $D_n \subset B_n$.  
Some computational housekeeping then shows that $x+yi \in U$.
\end{proof}

\begin{lemma} \label{small} If $a+bi \in B = ( B_{n+1} \cap Oct_n) \setminus(  B_n \cup (1+i) B_n)$, then $B_n \twoheadrightarrow \mathbb{Z}[i]/(a+bi)$.
\end{lemma}
\begin{proof} Proof by induction. 
Simple computations show this holds true for $n \in \{0, 1\}$, so for the rest of the proof, assume that $n \geq 2$.  
For ease of notation, we again define $U = \bigcup_{q \in \mathbb{Z}[i]} (B_n + q(a+bi))$.  
The set $B$ is closed under complex conjugation and multiplication by units, so as $(1 + i) \nmid a + bi$, 
we can assume without loss of generality that $w_n - 2 \geq a > b \geq 0$.
By applying Proposition \ref{triangle}, it suffices to show that $\mathscr{S}_{a+bi} \subset U$ to prove our claim.

As $0 <a \leq w_n -2$, the set $\mathscr{S}_{a +bi} \subset Oct_n$, so if $x +iy \in \mathscr{S}_{a+bi}$ and $(1+i) \nmid (x +yi)$, then $x +iy \in D_n \subset  B_n \subset U$.
For the rest of this proof, assume that $x +yi \in \mathscr{S}_{a+bi}$ and that $(1+i) | (x +yi)$; 
we must show that $x +yi \in U$.  We do this by  showing that either $x +yi \in B_n$ or 
 $x +yi \in Oct_n + u(a+bi)$ for some $ u \in \Z[i]^{\times}$, as then $x +yi \in U$ by Lemma \ref{oct_translate}.

Let us first consider $x +yi$, where $x, y \neq 0$.  
Suppose that $2^k \parallel \gcd(x,y)$, so that $2^k \leq  x,y < x+y \leq w_n -2^k $ (as $x +y < a \leq w_n -2$) 
and thus $2^k \leq x,y  \leq w_n - 2^{k+1}$.  
As $2^{k+1} < w_n$, we see by Lemma \ref{identities} that $k\leq \lfloor \frac{n}{2} \rfloor$ 
and that \[x + y \leq w_n - 2^k + (w_{n+1} - w_n - 2^k) = w_{n+1} - 2^{k+1}.\]
If $x + y \leq w_{n+1} - 3 \cdot 2^k$, then $x +yi \in 2^k S_{n-2k} \subset B_n \subset U$.

If not, then $x + y = w_{n+1} - 2^{k+1} < a \leq w_n -2$ and thus $w_{n+1} - 2^{k+1} \leq w_n - 2^k$.
We rearrange to see that $w_{n+1} - w_n \leq 2^k$ and thus $\lfloor \frac{n+1}{2} \rfloor \leq k$ by Lemma \ref{identities}.

In this situation, $\lfloor \frac{n+1}{2} \rfloor \leq k \leq \lfloor \frac{n}{2} \rfloor$, so $n = 2k$, $k \geq 1$, $a > x + y = 2^{k+1}$, and $x= y = 2^k$. 
We know that $2 \nmid \gcd (a-2^k, b-2^k)$, that 
$|a-2^k| , |b - 2^k| \leq w_n - 2^k - 2  < w_n -2$, and that 
\begin{align*}
|a-2^k| + |b-2^k| & \leq \max \{ a+b - 2^{k+1}, a-b\} \\
& \leq \max \{ w_{n+2} - 2^{k+1} - 3, w_n -3\}\\
& \leq w_{n+1} - 3,
\end{align*}
so $(a-x) + (b-y)i \in D_n \subset B_n$ and $x +yi \in U$. 

Now we consider $x+yi$, where one of the coordinates is zero.  
Label the non-zero coordinate $z$.  
If $2^k \parallel z$ and if $2^k \leq z \leq w_n - 2^{k+1}$, 
then $k \geq 1$ and $z \leq w_{n+1} - 3 \cdot 2^k$, demonstrating that $z, zi \in 2^k S_{n-2k} \subset B_n \subset U$.  
If $2^k \leq z = w_n - 2^k $, then $0 \leq b, |a-z| < w_n -2$.  As $2 \nmid \gcd (a-z, b)$ and 
\begin{align*}
0 < b + |a-z| &\leq \max (a +b-z, b + z-a)\\
&\leq \max (w_n + 2^k -3, w_n - 2^k -1)\\
& \leq w_{n+1} - 3,
\end{align*}, 
$(a-z) + bi \in D_n \subset B_n$, allowing us to conclude that both
$z , zi \in U$ by Lemma \ref{broom}.
\end{proof}

\subsection{Propositions at the heart of our proof}\label{meat}
Our main proof requires us to show that if $A_{\Z[i], j} = B_j$ for all $0 \leq j \leq n$, and if $a+bi \in B_{n+1} \setminus ( B_n \cup (1+i)\Z[i])$, then 
$\mathscr{S}_{a+bi} \subset U = \bigcup_{q \in \mathbb{Z}[i]} (B_n + q(a+bi))$.  
Lemma \ref{small} established our claim for the $a+bi \in B_{n+1} \setminus ( B_n \cup (1+i)\Z[i])$ that are also in $Oct_n$.
We now prove it for the $a+bi \in B_{n+1} \setminus ( B_n \cup (1+i)\Z[i])$ that are not in $Oct_n$.  
First, Proposition \ref{inside_the_octogon} shows that, under our assumptions, $\mathscr{S}_{a+bi} \cap Oct_n \subset U$.  
Proposition \ref{outside_the_octogon} then demonstrates that, under the same assumptions, $(\mathscr{S}_{a+bi}\setminus Oct_n) \subset U$ as well.  

\begin{prop}\label{inside_the_octogon}  Suppose that $A_{\Z[i], n} = B_n$.  
If $a +bi \in B_{n+1} \setminus (Oct_n \cup (1+i) \Z[i])$, if $a >b \geq 0$, and if $x+yi \in \mathscr{S}_{a+bi} \cap Oct_n$, then $x +yi \in U = \bigcup_{q \in \mathbb{Z}[i]} (B_n + q(a+bi))$. 
\end{prop}

\begin{proof}    

Suppose that $x +yi \in Oct_n$ and that $2^l \parallel \gcd(x,y)$. 
If $x +yi \in S_n \subset B_n$, then $x+yi$ is certainly an element of $U$, so we will assume for the rest of this proof that 
$x+yi \notin S_n$, so $1 \leq l \leq \lfloor \frac{n}{2} \rfloor$. 
Lemma \ref{max_power} states  that $l \leq \lfloor \frac{n}{2} \rfloor +1$.  
If $x+yi \in Oct_n \cap \mathscr{S}_{a+bi}$ and $l = \lfloor \frac{n}{2} \rfloor +1$, 
then $x +yi \in \{ 2^{k+1}, 2^{k+1} i\}$ when $n = 2k$, and $x \in \{2^{k+1}, 2^{k+1} i, 2^{k+1}(1+i) \}$ when $n = 2k+1$.  
Checking all five cases shows that at least one of $(a+bi) - (x+yi)$ and $i(a+bi) - (x+yi)$ must be an element of $B_n$.  
We therefore assume for the rest of the proof that $l \leq \lfloor \frac{n}{2} \rfloor$,
so $1\leq l \leq \lfloor \frac{n}{2} \rfloor$ and  $x +yi \in B_{n+1}$ by Corollary \ref{one_up}.   
Because $a > b \geq 0$ and $a+bi \notin Oct_n$, we observe
 that $a > w_n -2$.  

As $x +yi \in Oct_n$, we note that $x, y \leq w_n - 2^l$ and $x+y \leq w_{n+1} -\max(3,2^l)$.
Theorem \ref{octo_union} shows $x+yi \in B_n$ if and only if $x, y \leq w_n - 2^{l+1}$
and $x+y \leq w_{n+1} - 3 \cdot 2^l$.  Our element
$x+yi \in Oct_n \cap (B_{n+1} \setminus B_n)$ then falls into one of three cases:  either $x = w_n - 2^l$; $y = w_n - 2^l$;  
or $x,y \leq w_n - 2^{l+1}$ and $x+y \geq w_{n+1} - 2^{l+1}$.  We address each of the three cases below.

\underline{$\mathbf{x = w_n - 2^l}$:}  By our assumptions and Lemma \ref{identities}, 
\begin{align*}
0  \leq a -x  &\leq (w_{n+1}-2) - (w_n - 2^l)  \leq 2(w_{n+1} - w_n) -2  \leq w_n -2.\\
\intertext{
As $x+y \leq w_{n+1} - 2^l$, we also see that $y \leq w_{n+1} - w_n$.  
This then implies that }
|b-y| & \leq \max (b,y) \leq \max (w_n -2, w_{n+1} - w_n ) \leq w_n -2,\\
\intertext{and thus}
|a-x| + |b-y| & = \max ( a+ b - (x+y) , (a-b) + y -x) \\
& \leq \max ( w_{n+2} -3 - w_n + 2^l, w_{n+1} -3 + w_{n+1} - w_n - (w_n + 2^l) ) \\
& \leq \max (w_{n+1} - 3, 2(w_{n+1} - w_n)-2^l - 3 ) \\
&= w_{n+1} -3.
\end{align*}

We conclude that $(a-x) + (b-y)i \in Oct_n$ and thus $x+yi \in U$ by Lemma \ref{broom}.\\

\underline{$\mathbf{y = w_n - 2^l}$:}   
When $y = w_n -2^l$, then $0 \leq  a-y  \leq 2(w_{n+1} - w_n )- 2 \leq w_n -2.$
The condition $a-x > w_n -2$ is equivalent to  $b+x \leq a+b - w_n +1$;
the right hand side is bounded above by $w_{n+2} - 3 - w_n + 1 = w_n -2$.  
The assumption is also equivalent to $x < a-w_n +2$.
As $a-w_n +2 < w_{n+1} - w_n$, note that $x \leq w_{n+1} - w_n - 2^l$.  
We then see that if $a -x > w_n -2$, then
\begin{align*}
|a-y| + |b+x| &\leq a+b -y + x \\
&\leq w_{n+2} - 3 - w_n +2^l + w_{n+1} - w_n - 2^l \\
&= w_{n+1} - 3,
\end{align*}
demonstrating that $(a-y) + (b+x)i \in Oct_n$.  

Similarly, if $b+ x \leq y = w_n - 2^l \leq w_n -2$, then the odd sum $|a-y| + |b+x| \leq a -y +y  =a \leq w_{n+1} -2$, so $|a-y| + |b+x| \leq w_{n+1} -3$ and 
$(a-y) + (b+x) i \in Oct_n$. Lemma \ref{broom} shows that $x +yi \in U$ when either $a-x > w_n -2$ or $b+x \leq y$.

Let us now suppose that $a-x \leq w_n -2$ and $b+x >y$.  Note that
$|b-y| \leq  w_n -2$.  If $b \geq y$, then 
\begin{align*}
|a-x| + |b-y| &= (a+b) - (x+y) \leq w_{n+2} - 3 - w_n + 2^l \leq w_{n+1} - 3; \\
\intertext{otherwise, $b < y < b+x$ and }
|a-x| + |b-y| &= a + (y - (b+x)) \leq a-1 \leq w_{n+1} - 3.
\end{align*}
Either way, $(a-x) + (b-y) i \in Oct_n$ and thus $x+yi \in U$ by Lemma~\ref{broom}.\\

\underline{$\mathbf{x,y \leq w_n - 2^{l +1} \text{ and } x + y \geq w_{n+1} - 2^{l+1}}$:}
These conditions imply that $|b-y| \leq w_n -2$, that $\min(x,y) \geq w_{n+1} - w_n$, and that 
\[w_{n+1} - w_n < a-x, a-y \leq w_{n+1} -2 - (w_{n+1} -w_n) = w_n -2.\]
If $b \geq y$, then 
\[|a -x| + |b-y| = (a+b) - (x+y) \leq w_{n+2} - 3 - w_{n+1} + 2^{l+1} = w_{n+1} - 3\]
 and $(a-x) + (b-yi) \in Oct_n$ by Lemma \ref{identities}, as desired.

If $b + x \leq y \leq w_n -2$, then  $|a-y| + |b+x| \leq a- y + y  \leq w_{n+1} -2$ and thus the odd sum $|a-y| + |b+x|$ is bounded above by $w_{n+1} -3$, showing that 
$(a-y) + (b+x) i \in Oct_n$.

We are then left with when $b+x > y > b$, implying that 
\[|a-x| + |b-y| = a +y - (b+x) \leq a-1 \leq w_{n+1} - 3,\]
demonstrating that $(a - x ) + (b-y) i \in Oct_n$.  In all three scenarios, $x +yi \in U$ by Lemma \ref{broom}.

\end{proof} 

\begin{prop} \label{outside_the_octogon}  Suppose that $A_{\Z[i], n} = B_n$.  
If $a+bi \in B_{n+1} \setminus (Oct_n \cup (1+i)\Z[i])$, if $a > b \geq 0$, and if $x +yi \in \mathscr{S}_{a+bi} \setminus Oct_n$, then 
$x +yi \in U = \bigcup_{q \in \mathbb{Z}[i]} (B_n + q(a+bi))$. 
\end{prop}

\begin{proof} Our assumptions imply that $b \leq w_n -2 <a$.   As $x +yi \in \mathscr{S}_{a+bi} \setminus Oct_n$,
$x +y \leq a-1 \leq w_{n+1} -3$, so  either $x > w_n -2$ or $y > w_n -2$.
We address the two cases below.

\underline{$\mathbf{x > w_n -2}:$} As $x+yi \in \mathscr{S}_{a+bi}$, our bound implies that 
\[\max (y, 2^l) \leq a-x \leq w_{n+1} - w_n -1< w_n -2 < x.\]
Suppose that $2^l \parallel (a-x, b-y),$
so that 
\begin{equation}\label{heart}
0 \leq y< a-x \leq w_{n+1} - w_n - 2^l < 2(w_{n+1} - w_n - 2^l) \leq w_n - 2^{l+1},
\end{equation}
and $l \leq \lfloor \frac{n}{2} \rfloor$ by Lemma \ref{identities}.
If $|b-y| \leq w_n - 2^{l+1}$, then 
\[|a-x| + |b-y|  \leq (w_{n+1} - w_n - 2^l) + (w_n - 2^{l+1}) = w_{n+1} - 3\cdot 2^l,\]
and $(a-x) + (b-y)i \in 2^l S_{n-2l} \subset B_n,$ so Lemma \ref{broom} places $x +yi \in U$.

If $|b-y| > w_n - 2^{l+1}$, then $b-y = w_n -2^l$, as $0 \leq y <  w_n - 2^{l+1}$ and $0\leq b \leq w_n -2$, thereby forcing $l \geq 1$.
Lemma \ref{identities} then shows that, as $l\leq \lfloor \frac{n}{2} \rfloor$,  
\begin{align*}
\max (x, 2^l) \leq a - b + y & \leq (w_{n+1} -2) - (w_n - 2^{l})  \leq 2(w_{n+1} -w_n) -2 \leq w_n -2,\\
\intertext{that}
0 < a +b - x &\leq (w_{n+2} -3) - (w_n -1) = w_n -2,\\
\intertext{and that}
|a-b+y| + |a+b -x| & = (a+b) + (a-x) -(b-y) \\
&\leq (w_{n+2} -3) + (w_{n+1} - w_n -2^l) - (w_n -2^l) \\
&= w_{n+1} -3.
\end{align*}
We noted previously that $l \geq 1$, so $2 | (a-x) + (b-y)i$.  As $(1+i) \nmid (a+bi)$, it follows that $(1+i) \nmid (x+yi)$ and thus $(1+i)$ does not divide
$(1+i)(a+bi) - i(x+yi) = (a-b+y) + (a+b-x)i$.
We conclude that 
$(a-b+y) + (a+b-x) i \in D_n \subset B_n$ and thus
$x +yi \in (B_n + (1-i)(a+bi)) \subset U$.

\underline{$\mathbf{y > w_n -2}:$} 
Suppose that 
$2^l \parallel (a-y, b+x)$.  We apply Lemma \ref{identities} to see that
\begin{equation}\label{med}
0 < a-y \leq w_{n+1} - w_n - 2^l < 2(w_{n+1} - w_n - 2^l) \leq w_n - 2^{l+1},
\end{equation}
and $l \leq \lfloor \frac{n}{2} \rfloor.$
If $b+x \leq w_n - 2^{l+1}$, then 
\[|a-y| + |b+x| \leq (w_{n+1} - w_n - 2^l) + (w_n - 2^{l+1}) = w_{n+1} - 3\cdot 2^l\]
and $(a- y) + (b+x)i \in 2^l S_{n - 2l} \subset B_n$, setting
$x +yi \in  U$ by Lemma \ref{broom}.

If $b + x > w_n - 2^{l+1}$, then 
\begin{equation}\label{needed?}
w_n - 2^l \leq b+x < b + (a-y)  \leq w_n -2 <a,
\end{equation}
and $l \geq 2$. Equation \ref{needed?} just showed that $0 < a+b - y \leq w_n -2$, so as
\begin{align*}
|a - b-x| = a - (b+x) & \leq w_{n+1} - 2 - (w_n - 2^l) \leq  w_n -2\\
\intertext{and}
|a -b-x| + |a + b - y| & \leq (a -y) + (a +b) - (b+x) \\
&\leq (w_{n+1} - w_n - 2^l) + (w_{n+2} - 3) +(2^{l} - w_n) \\
&=w_{n+1} -3,
\end{align*}
we see that $(a-b -x ) + (a+b -y )i \in Oct_n$.  
As $l \geq 2$, $(1 +i)$ divides $(a-y) +(b+x)i = (a+bi) + i(x+yi)$.
We deduce that $(1 +i) \nmid (x+yi)$, and thus  $(1+i)$ does not divide
$(a -b-x) + (a+b -y)i = (1+i)(a+bi) - (x+yi)$.  
We conclude that $(a-b-x) + (a+b-y)i \in D_n \subset B_n$ and that 
$x+yi \in (B_n + (1+i)(a+bi)) \subset U$.
\end{proof}

\subsection{Main Results}\label{finally}
\begin{theorem} (Lenstra, \cite{Lenstra})\label{main_result} For $n \geq 0$, $A_{\mathbb{Z}[i],n} = \phi_{\Z[i]}^{-1}([0,n])= B_n$.
\end{theorem} 

\begin{proof}  Proof by induction.  
Example \ref{example_in_G} computes our base cases and shows that $A_{\mathbb{Z}[i],n} = B_n$ when $n =0,1,$ and $2$.  
Suppose that $n \geq 2$ and $A_{\mathbb{Z}[i],j} = B_j$ for all $j < n$.   
If $(B_n \setminus (1+i)\Z[i]) \subset A_{\Z[i],n}$, then $A_{\mathbb{Z}[i],n} = B_{n}$ by Lemma \ref{subset_containment}.
It is clear that if $a + bi \in B_{n-1} = A_{\Z[i], n-1}$, then $a +bi \in A_{\Z[i], n}$. 
To prove our theorem, it therefore suffices to prove that if $a + bi \in B_n \setminus (B_{n-1} \cup (1+i) \Z[i])$, then $a + bi \in A_{\Z[i], n}$.  

Lemma \ref{small} shows that if $a+bi \in B_n \setminus (B_{n-1} \cup (1+i) \Z[i])$ and $a + bi \in Oct_{n-1}$, then 
$B_{n-1} \twoheadrightarrow \Z[i]/(a+bi)$. 
As $B_{n-1} = A_{\Z[i], n-1}$, $a+bi \in A_{\Z[i], n}$.

If $a + bi \notin Oct_{n-1}$, it is certainly not in $B_{n-1}$, so the set of $a+bi \in B_n \setminus (B_{n-1} \cup (1+i) \Z[i])$ that are not in $Oct_{n-1}$ 
is the set $B_n \setminus (Oct_{n-1} \cup (1+i) \Z[i])$. Suppose that $a + bi \in B_n \setminus (Oct_{n-1} \cup (1+i) \Z[i])$, that $\alpha = \max (|a|, |b|)$, and that 
$\beta = \max (|a|, |b|)$.  
As $\alpha > \beta \geq 0$, Proposition  \ref{inside_the_octogon} says that 
$\mathscr{S}_{\alpha + \beta i} \cap Oct_{n-1} \subset U = \bigcup_{q \in \Z[i]} (B_{n-1} + q (a+bi))$ and 
Proposition \ref{outside_the_octogon} says that $\mathscr{S}_{\alpha + \beta i} \setminus Oct_{n-1} \subset U$.
The union $\mathscr{S}_{\alpha + \beta i} \subset U$ and $B_{n-1}$ is closed under multiplication by units, so 
$B_{n-1} = A_{\Z[i], n-1} \twoheadrightarrow \Z[i]/(\alpha + \beta i)$ by Lemma \ref{triangle}.  
As $\alpha + \beta i \in A_{\Z[i], n}$, $a+bi \in A_{\Z[i], n}$ by Corollary \ref{you_get_the_whole_set}.
We have now shown that $B_n \setminus (B_{n-1} \cup (1+i) \Z[i]) \subset A_{\Z[i],n}$, as required.  
\end{proof}

We can now prove Theorem \ref{pre-images} and describe the sets $\phi_{\Z[i]}^{-1}(n)$.
\begin{proof} (of Theorem \ref{pre-images})
As Theorem \ref{main_result} shows that $\phi_{\Z[i]}^{-1}([0,n]) = B_n$, it follows that, for $n \geq 1$,
\begin{align*}
\phi_{\Z[i]}^{-1}(n) &= B_n \setminus B_{n-1}\\
& = \coprod_{j=1}^{\lfloor n/2 \rfloor} 2^j S_{n-2j}  \setminus \left (\coprod_{j=0}^{\lfloor (n-1)/2 \rfloor} 2^j S_{n-2j-1} \right ).
\end{align*}
Then, for $k \geq 0$,
\begin{align*}
\phi_{\Z[i]}^{-1}(2k+1) &= B_{2k+1} \setminus B_{2k}\\
& = \coprod_{j=1)}^{\lfloor n/2 \rfloor} 2^j ( S_{2(k-j)+1} \setminus S_{2(k-j)}) \\
& = \displaystyle \coprod _{j=0}^{k}
\left ( 
a+bi:
\begin{array}{c}
2^j \parallel (a+bi); |a|, |b|\leq w_n - 2^{j+1}; \\
|a| + |b| \leq w_{n+1} - 3 \cdot 2^j ,\\
\text{ and either } \max(|a|, |b|) > w_{n-1} - 2^{j+1} \\
\text{ or } |a| + |b| > w_{n} - 3 \cdot 2^j 
\end{array}
\right )\\
\intertext{ and for $k \geq 1$,}
\phi_{\Z[i]}^{-1}(2k) &= B_{2k} \setminus B_{2k-1}\\
& = (2^k S_0) \cup \coprod_{j=1)}^{\lfloor n/2 \rfloor} 2^j ( S_{2(k-j)+1} \setminus S_{2(k-j)}) \\
& = \begin{array}{c}
 \{\pm 2^k, \pm 2^k i \} \cup \\
 \displaystyle \coprod _{j=0}^{k-1}
\left ( 
a+bi:
\begin{array}{c}2^j \parallel (a+bi); |a|, |b|\leq w_n - 2^{j+1};\\
  |a| + |b| \leq w_{n+1} - 3 \cdot 2^j ,\\
\text{ and either } \max(|a|, |b|) > w_{n-1} - 2^{j+1} \\
\text{ or } |a| + |b| > w_{n} - 3 \cdot 2^j 
\end{array}
\right ).
\end{array}
\end{align*}
\end{proof}

\section{Application: Answering Samuel's question}\label{Application}

As mentioned in Sections~\ref{introduction} and \ref{history}, Pierre Samuel computed $|\phi_{\Z[i]}^{-1} (n)|$ for $n \in [0,8]$ (\cite{Samuel}, p. 290).  
He did not compute $|\phi_{\Z[i]}^{-1}(9)|$, presumably because the sets involved became so large that the computations became unwieldy.  
After all, $|\phi_{\Z[i]}^{-1}(8)| = 3364$ and $A_{\Z[i],8} = 6457$ (see Table).

In this section, we will describe the naive method to find $|\phi_{\Z[i]}^{-1}(9)|$ using techniques known when Samuel wrote his his survey.
Then we will describe the (still exponential) techniques implied by Lenstra's theorem to compute $|\phi_{\Z[i]}^{-1}(9) |$.
Lastly, we present a closed-form exponential function that computes $|\phi_{\Z[i]}^{-1}(9) |$.
Appendix A is a table  presenting 
$|\phi_{\Z[i]}^{-1}(n) |$ and $|A_{\Z[i], n}|$ for $n \in [0,\ldots, 20]$
and Appendix B contains Sage code used to do this section's calculations.  
To clarify, the last subsection introduces a closed-form exponential function; the previous subsections require doing exponentially many operations.

\subsection{Before Lenstra}

We present a reasonable method to calculate $|\phi_{\Z[i]}^{-1}(9) |$ with the knowledge Samuel had when he wrote his survey \cite{Samuel}.
He had computed $|\phi_{\Z[i]}^{-1}(n) |$ for $n \in [0, \ldots, 8]$, so he knew that $|A_{\Z[i],8}| = 6457$.  
He also knew that if $a + bi \in \phi_{\Z[i]}^{-1}(9) $, then $\Nm(a+bi) \leq 6457$, 
as every equivalence class in $\Z[i]/(a+bi)\Z[i]$ must have a representative in $A_{\Z[i],8}$.
In order to find $|\phi_{\Z[i]}^{-1}(9) |$, he would have  had to examine each element of norm $\leq 6457$, 
and see if all of their cosets had a representative in $A_{\Z[i], 8}$.

We reduce our study to pairs $a + bi$ such that $a \geq b \geq 0$, as that cuts our search range by approximately a factor of $8$.  
A simple program in SAGE (not available in 1971)
shows that $|\{a+bi \in \Z[i]: 0 \leq b \leq a, \Nm(a+bi) \leq 6457 \} | = 2605$ (see Appendix B's first listing).
We then go through this list and remove all elements that are already in $A_{\Z[i],8}$.  
Appendix B's second program shows there are $842$ elements $a+bi \in A_{\Z[i],8}$ such that 
$0 \leq b \leq a$, so we would have to examine $1763$ elements (see Appendix B's third program).
For each of these $1763$ remaining $a+bi$, we would have to check whether every elment in the associated set 
$S \cup T$ (see Lemma \ref{two_squares}) is congruent to some element of $A_{\Z[i],8}$ modulo $a+bi$.  
This means checking $7476972$ cosets against $6457$ elements.

\subsection{Using Lenstra's Theorem}

Lenstra's Theorem makes it significantly easier to study $|\phi_{\Z[i]}^{-1}(9) |$.
Every element of $A_{\Z[i],9} \setminus A_{\Z[i],8}$ can be written as $u(1+i)^9 +b$ for some $u \in \{ \pm 1, \pm i \}$ and some $b \in A_{\Z[i],8}$.
A simple way to find $|\phi_{\Z[i]}^{-1}(9) |$ would be to compute all $4 \cdot 6457 = 25828$ sums $\{ u(1+i)^9 + b, u \in \Z[i]^{\times}, b \in A_{\Z[i],8} \}$,
remove all dulplicate elements from the list, and then remove any elements that are also in $A_{\Z[i], 8}$.  
There are ways to make the general computation more efficient, but they all involve calculating $\sim |A_{\Z[i],n}|$ sums, where $c$ is a small constant.  
Appendix \ref{Table}'s table shows that this involves exponentially (in $n$) many sums.

\subsection{Explicit Formula}

Computing $|\phi_{\Z[i]}^{-1}(9) |$ is the same as calculating $|A_{\Z[i],9} \setminus A_{\Z[i],8}| = |A_{\Z[i],9}| - |A_{\Z[i],8}|$.
Theorem \ref{octo_union} shows that each $A_{\Z[i],n} \setminus 0$ can be written as a disjoint union of multiples of sets $S_j$, so to find 
$|B_n|$, we need to know $|S_n|$.

\begin{lemma}  For $n \geq 1$, $|S_n| = 3(w_n -2)^2 + 2(w_n -2) -6(w_n - w_{n-1})(w_n - w_{n-1} -1)$.
\end{lemma}
\begin{proof}
By symmetry,
\begin{equation*}
\begin{split}
|S_n| = {}& 4 | \{ x \in \Z: 1 \leq x \leq w_n -2, 2 \nmid x\} \\
 &  + 4 | \{ x+yi \in \Z[i]: 1 \leq x,y \leq w_n -2, x + y \leq w_{n+1} - 3, 2 \nmid \gcd(x,y) \}\\
 ={}& 4 \left (\frac{w_n -2}{2} \right ) + 4 | \{x + yi \in \Z[i]: 1 \leq x, y \leq w_n -2; 2 \nmid \gcd (x,y) \}|\\
  &  - 4 |\{x+yi\in \Z[i]: w_{n+1} - 2 \leq x+y; w_{n+1} - w_n \leq x, y\leq w_n -2; 2 \nmid \gcd(x,y) \} |\\
   ={}&   + 4 | \{x + yi \in \Z[i]: 1 \leq x, y \leq w_n -2 \}|\\
   & -4 | \{x + yi \in \Z[i]: 1 \leq x, y \leq w_n -2; 2\mid x; 2 \mid y \}|\\
 &  - 4 \sum_{\mathclap{\substack{x = w_{n+1} - w_n \\x \text{ odd} }}}^{w_n -2} | \{y: w_{n+1} -2 -x \leq y \leq w_n -2 \} | \\
 &   - 4  \sum_{\mathclap{\substack{x = w_{n+1} - w_n \\x \text{ even} }}}^{w_n -2} | \{y: 2 \nmid y, w_{n+1} -2 -x \leq y \leq w_n -2 \}| \\
   ={}& 4 \left (\frac{w_n -2}{2} \right ) + 4 (w_n -2)^2 -4 \left ( \frac{w_n -2}{2} \right )^2 
  - 4 \sum_{\mathclap{\substack{x = w_{n+1} - w_n \\x \text{ odd} }}}^{w_n -2} x - (w_{n+1} - w_n) +1 \\
 &   - \frac{4}{2}  \sum_{\mathclap{\substack{x = w_{n+1} - w_n \\x \text{ even} }}}^{w_n -2} x - (w_{n+1} -w_n) \\
  ={}& 3(w_n -2)^2 + 2(w_n -2)   - 4 \hspace{-.7 cm}\sum_{\mathclap{\substack{x = 0 \\x \text{ odd} }}}^{w_{n+2} -w_{n+1} -3} \hspace{-.7 cm}x - (w_{n+1} - w_n) +1 
  - 2 \hspace{-.7 cm} \sum_{\mathclap{\substack{x = 0 \\x \text{ even} }}}^{w_{n+2} -w_{n+1} -2} \hspace{-.7 cm} x \\
  ={}& 3 (w_n -2)^2 + 2(w_n -2)  -6  \sum_{\substack{ x = 0\\x \text{ even}}}^{\mathclap{w_{n+2} - w_{n+1} -2}} x\\
  ={}& 3 (w_n -2)^2 + 2(w_n -2)  -12  \sum_{x = 0}^{\mathclap{w_n - w_{n-1} -1}} x\\
  ={}& 3 (w_n -2)^2 + 2(w_n -2)  -6 \cdot 2  \sum_{x = 0}^{\mathclap{w_{n} - w_{n-1} -1}} x\\
  ={}& 3 (w_n -2)^2 + 2(w_n -2) -6 (w_n - w_{n-1})(w_n - w_{n-1} -1).
 \end{split}
 \end{equation*}
 \end{proof}
 
 \begin{coro}
 If $n = 2k +1$, $k \geq 0$, then $S_n| = 42 \cdot 4^k - 34 \cdot 2^k + 8$.
 If $n = 2k$, $k \geq 1$, then $|S_n| = 21 \cdot 4^k - 24 \cdot 2^k + 8$.
 \end{coro}  
 
 We can now use our formula for $|S_n|$ to find $|A_{\Z[i],n}|$.
 
 \begin{theorem}\label{pre-image_cardinality}  For all $k \geq 0$, $|A_{\Z[i], 2k+1}| = 14 \cdot 4^{k+1} - 34 \cdot 2^{k+1} + 8k + 29$.  
 For $k \geq 1$, $|A_{\Z[i], 2k}| = 28 \cdot 4^{k} - 48 \cdot 2^{k} + 8k + 25$.
 \end{theorem} 
 \begin{proof}
 Theorem \ref{octo_union} shows that $A_{\Z[i],n} \setminus 0 = \coprod_{j=0}^{n/2} 2^j S_{n -2j}$, so \\
 $|A_{\Z[i],n}|= 1 + \sum_{j=0}^{n/2} |S_{n-2j}|$.  
 Therefore 
 \begin{equation*}
 \begin{split}
 |A_{\Z[i],2k}| &= 1 + |S_0| + \sum_{j=1}^k |S_{2j}|\\
 &=5 + \sum_{j=1}^k 21 \cdot 4^j - 24 \cdot 2^j + 8 \\
 &= 5 + \sum_{j=0}^{k-1} 84 \cdot 4^j - 48 \cdot 2^j+ 8\\
 &=84 \left (\frac{ 4^{k}-1}{3} \right ) - 48 \cdot 2^{k} + 8k + 53 \\
 &= 28 \cdot 4^{k} - 48 \cdot 2^{k} + 8k + 25
 \end{split}
  \end{equation*}
 and 
 \begin{equation*}
 \begin{split}
 |A_{\Z[i],2k +1}| &= 1  + \sum_{j=0}^k |S_{2j +1}|\\
 & = 1 + \sum_{j=0}^k 42 \cdot 4^j - 34 \cdot 2^j +8\\
 & = 1 + \left (42 \left (\frac{ 4^{k+1}-1}{3} \right ) - 34 \cdot 2^{k} + 8k \right ) \\
 & = 14 \cdot 4^{k+1} - 34 \cdot 2^{k+1} + 8k + 29.
 \end{split}
 \end{equation*}
 \end{proof}
 
 This naturally leads to Theorem \ref{size_of_sets}.  
 \begin{proof} (Of Theorem \ref{size_of_sets}).  
Applying Theorem \ref{pre-image_cardinality} reveals that
 \begin{align*}
 |\phi_{\Z[i]}^{-1}(2k)| & = | A_{\Z[i], 2k} \setminus A_{\Z[i], 2k-1}|\\
 & = | A_{\Z[i], 2k}| - |A_{\Z[i],2k -1}|\\
 & = (28 \cdot 4^{k} - 48 \cdot 2^{k} + 8k + 25) - (14 \cdot 4^{k} - 34 \cdot 2^{k} + 8(k-1) + 29 )\\
 & = 14 \cdot 4^k -14 \cdot 2^k +4\\
 \intertext{ and }
 |\phi_{\Z[i]}^{-1}(2k+1)| & = | A_{\Z[i], 2k+1} \setminus A_{\Z[i], 2k}|\\
 & = | A_{\Z[i], 2k+1}| - |A_{\Z[i],2k }|\\
 & = (14 \cdot 4^{k+1} - 34 \cdot 2^{k+1} + 8k + 29) - (28 \cdot 4^{k} - 48 \cdot 2^{k} + 8k + 25)\\
 & = 28 \cdot 4^k - 20 \cdot 2^k +4.
 \end{align*}
 \end{proof}
 
 \newpage 
 \appendix
 \section{Table}\label{Table}
 
 \begin{table}[h!]
 \centering
 \begin{tabular}{|c c c c|}
 \hline
 $n$ & $|S_n|$ & $|B_n|$ & $|\phi_{\Z[i]}^{-1}(n)|$\\
 \hline 
 0& 4&				5&			5\\
 1& 16&				17&			12\\
 2& 44&				49&			32\\
 3& 108&			125&		76\\
 4& 248& 			297&		172\\
 5& 544&			669&		372\\
 6& 1160&			1457&		788\\
 7& 2424& 			3093&		1636\\
 8& 5000&			6457&		3364\\
 9& 10216&			13309&		6852\\
 10& 20744&			27201&		13892\\
 11& 41928&			55237&		28036\\
 12& 84488&			111689&		56452\\
 13& 169864&		225101&		113412\\
 14& 341000&		452689&		227588\\
 15& 683784&		908885&		456196\\
 16&  1370120&		1822809&	913924\\
 17& 2743816&		3652701&	1829892\\
 18& 5492744&		7315553&	3662852\\
 19& 10992648&		14645349&	7329796\\
 20& 21995528&		29311081&	14665732\\
 21& 44005384&		58650733&	29339652\\
 22& 88031240 &		117342321&	58691588\\
 23& 176091144 &	234741877&	117399556\\
 24& 352223240 &	469565561&	234823684\\
 25& 704503816 &	939245693&	469680132\\
 \hline
 \end{tabular}
 \caption{$|S_n|$, $|B_n|$, and $|\phi_{\Z[i]}^{-1}(n)|$ for $n \in [0,25]$}
 \label{table:1}
\end{table}

\section{Code}\label{Code}
Finding the set of $x +yi$ pairs where $0 \leq y \leq x , x^2 + y^2 \leq 6456$:\\
\begin{lstlisting}[numbers=left,numberstyle=\tiny,numbersep=0pt]
	sage: Norm_set =[]
	sage: for x in [0,..,80]:  
			f = min(x,floor(sqrt(6457-x^2)))  
			for y in [0,..,f]:
				Norm_set.append((x,y))  
	sage:  len(Norm_set)
\end{lstlisting}   

Here is the code to compute the elements in $B_8$ such that $0 \leq b \leq a$.
Recall that 
$B_8 = S_8 \cup 2 \cdot S_6 \cup 4\cdot S_4 \cup 8\cdot S_2 \cup 16\cdot S_0.$
We compute the part of $16 \cdot S_0$ such that $0 \leq b \leq a $, 
and then repeat with $8 \cdot S_2$, etc. \\
\begin{lstlisting}[numbers=left,numberstyle=\tiny,numbersep=0pt]
sage: B=[(0,0)]
	
sage: for x in [1,..,1]:
		for y in [0,..,x]:
			if (gcd(x,y))%2==1:
				if (x +y) <=1:
					B.append((16*x,16*y))
sage: for x in [1,..,4]:
		for y in [0,..,x]:
			if (gcd(x,y))%2==1:
				if (x +y) <=5:
					B.append((8*x,8*y))
sage: for x in [1,..,10]:
		for y in [0,..,x]:
			if (gcd(x,y))%2==1:
				if (x +y) <=13:
					B.append((4*x,4*y))
sage: for x in [1,..,22]:
		for y in [0,..,x]:
			if (gcd(x,y))%2==1:
				if (x +y) <=29:
					B.append((2*x, 2*y))
sage: for x in [1,..,46]:
		for y in [0,..,x]:
			if (gcd(x,y))%2==1:
				if (x +y) <=61:
					B.append((x,y))
sage: print(len(B))
\end{lstlisting}

Computing the number of elements $x +yi$ pairs where $0 \leq y \leq x , x^2 + y^2 \leq 6456$ and $x +yi $ is not in $B_8$:\\
\begin{lstlisting}[numbers=left,numberstyle=\tiny,numbersep=0pt]
sage: Norm_List=[]  
sage: for x in [0,..,len(Norm_set)-1]:
		if (Norm_set[x]  in B)==0:
			Norm_List.append(Norm_set[x])    
sage: print(len(Norm_set)- len(B))
\end{lstlisting}

Computing the sizes of the sets $S_n$ for $0 \leq n \leq 21$:\\
\begin{lstlisting}[numbers=left,numberstyle=\tiny,numbersep=0pt]
sage: S_size=[4, 16]
sage: for n in [1,..,12]:
		S_size.append(21*4^n - 24*2^n + 8)
		S_size.append(42*4^n - 34*(2^n) +8)
sage: print(S_size)
\end{lstlisting}

Computing the sizes of the sets $B_n$ for $0 \leq n \leq 21$:\\
\begin{lstlisting}[numbers=left,numberstyle=\tiny,numbersep=0pt]
sage:	B_size=[0,..,21]
sage:	B_size=[0,..,21]
sage:	for n in [0,..,10]:
			B_size[2*n]=1
			for m in [0,..,n]:
				B_size[2*n] = B_size[2*n] + S_size[2*m]
			B_size[2*n +1]=1
				for m in [0,..,n]:
					B_size[2*n+1] = B_size[2*n +1]
									+ S_size[2*m +1]
sage: 	print(B_size)
\end{lstlisting}

Checking that the formula does give us the size of $B_n$ for $0 \leq n \leq 21$:\\
\begin{lstlisting}[numbers=left,numberstyle=\tiny,numbersep=0pt]
sage: formula_list=[0,..,25]
sage: formula_list[0]=5
sage: formula_list[1] =17
sage: for n in [1,..,12]:
		formula_list[2*n] = (28)*(4^(n)) - 48*2^(n) + 8*n +25
		formula_list[2*n +1] = (14)*(4^(n+1)) - 34*2^(n +1) 
								+8*n +29
sage: print(formula_list)        
\end{lstlisting}

Computing the sizes of the sets $\phi_{\Z[i]}^{-1}n$ for $1 \leq n \leq 21$:\\
\begin{lstlisting}[numbers=left,numberstyle=\tiny,numbersep=0pt]
sage:	B_size=[0,..,21]
sage:	for n in [0,..,10]:
			B_size[2*n]=1
			for m in [0,..,n]:
				B_size[2*n] = B_size[2*n] + S_size[2*m]
			B_size[2*n +1]=1
				for m in [0,..,n]:
					B_size[2*n+1] = B_size[2*n +1] 
					+ S_size[2*m +1]
sage: 	print(B_size)
\end{lstlisting}

Checking that the formula does give $\phi_{\Z[i]}^{-1}n$ for $2 \leq n \leq 21$:\\
\begin{lstlisting}[numbers=left,numberstyle=\tiny,numbersep=0pt]
sage:	for n in [1,..,12]:
			print(2*n, 14*4^n - 14*2^n +4)
			print(2*n+1, 28*4^n - 20*2^n + 4)
\end{lstlisting}

\section*{Acknowledgements}
I would like to thank my very patient spouse, Loren LaLonde, who has listened to me talk about this problem for the last fifteen years, and who has ensured that my LaTeX always compiled.

I greatly appreciate the help from Michael Bridgland, who not only read Martin Fuch's thesis \cite{Fuchs} for me (I don't speak German), 
but who also asked a question that led to Lemma \ref{max_power}.  
My frequent collaborator Jon Grantham gave me very useful comments on an earlier draft, leading to both the history and application sections.  

I would also like to thank H.W. Lenstra, Jr. for his helpful feedback and Franz Lemmermeyer for his help with this paper's background research and 
literature review.  I highly recommend Lemmermeyer's survey, ``The Euclidean Algorithm in Algebraic Number Fields,'' to anyone interested in the subject \cite{Lemmermeyer}.

\end{document}